\documentclass{amsart}

\usepackage{amsmath}
\usepackage{amsthm}
\usepackage{amssymb}
\usepackage{amsxtra}
\usepackage{mathrsfs}
\usepackage{lmodern}
\usepackage[utf8]{inputenc}
\usepackage[T1]{fontenc}
\usepackage{enumerate}
\usepackage{amscd}
\usepackage{color}
\definecolor{blue-black}{rgb}{0,0,0.8}
\definecolor{green-black}{rgb}{0,0.75,0}
\definecolor{red-black}{rgb}{0.8,0,0}
\usepackage{todonotes}
\usepackage{mathtools}
\usepackage[colorlinks,
             hyperindex,
             linkcolor=blue-black,
             urlcolor=red-black,
             citecolor=green-black,
             pdfproducer={pdfLaTeX},
             pdfpagemode=None,
             bookmarksopen=true
             bookmarksnumbered=true]{hyperref}
\usepackage[alphabetic]{amsrefs}
\usepackage[capitalize]{cleveref}

\theoremstyle{plain}
\newtheorem{thm}{Theorem}[section]
\newtheorem{prop}[thm]{Proposition}
\newtheorem{lem}[thm]{Lemma}
\newtheorem{cor}[thm]{Corollary}

\theoremstyle{definition}

\newtheorem{rem}[thm]{Remark}
\newtheorem{dfn}[thm]{Definition}
\newtheorem{conv}[thm]{Convention}

\numberwithin{equation}{section}

\newcommand\fb{\mathfrak b}
\newcommand\fg{\mathfrak g}
\newcommand\fh{\mathfrak h}
\newcommand\fk{\mathfrak k}
\newcommand\fl{\mathfrak l}
\newcommand\fn{\mathfrak n}
\newcommand\fp{\mathfrak p}
\newcommand\fq{\mathfrak q}
\newcommand\fs{\mathfrak s}
\newcommand\fu{\mathfrak u}

\newcommand\cC{\mathscr C}
\newcommand\cF{\mathscr F}
\newcommand\cH{\mathscr H}
\newcommand\cO{\mathscr O}
\newcommand\cP{\mathscr P}
\newcommand\cQ{\mathscr Q}
\newcommand\cS{\mathscr S}

\newcommand\bbC{\mathbb C}

\newcommand\bbP{\mathbb P}
\newcommand\bbQ{\mathbb Q}

\newcommand\bbZ{\mathbb Z}

\newcommand\bj{\mathbf j}
\newcommand\bk{\mathbf k}

\newcommand{\pplus}{\cP^+}
\newcommand{\qplus}{\cQ^+}
\newcommand{\uqg}{U_q(\fg)}
\newcommand{\uqk}{U_q(\fk)}
\newcommand{\uql}{U_q(\fl)}
\newcommand{\uqsltwo}{U_q(\fs\fl_2)}
\newcommand{\uqslfour}{U_q(\fs\fl_4)}
\newcommand{\uqbm}{U_q(\fb_-)}
\newcommand{\uqnp}{U_q(\fn_+)}
\newcommand{\uqnm}{U_q(\fn_-)}
\newcommand{\peq}{\preceq}

\newcommand{\ext}{\Lambda}
\newcommand{\extq}{\Lambda_q}
\newcommand{\symq}{S_q}
\newcommand{\sym}{S}
\newcommand{\of}{\cO_f}
\newcommand{\up}{\fu_+}
\newcommand{\um}{\fu_-}
\newcommand{\upm}{\fu_\pm}
\newcommand{\rhat}{\widehat{R}}
\newcommand{\Wl}{W_{\fl}}
\newcommand{\wl}{w_{\fl}}
\newcommand{\wzl}{w_{0,\fl}}
\newcommand{\wz}{w_0}
\newcommand{\grf}{\gr^\cF}
\newcommand{\zp}{\bbZ_+}
\newcommand{\ox}{\bar{x}}
\newcommand{\gc}{\Gamma^\circ}
\newcommand{\Fc}{\cF^\circ}
\newcommand{\grfc}{\gr^{\Fc}}
\newcommand{\oy}{\bar{y}}
\renewcommand{\rhd}{\triangleright}
\newcommand{\clifq}{\clif_q}
\newcommand{\cliff}{\clif}
\newcommand{\delbar}{\mkern2mu\overline{\mkern-2mu\partial}}
\newcommand{\eqdef}{:=}
\DeclareMathOperator{\tr}{tr}
\DeclareMathOperator{\ad}{ad}
\DeclareMathOperator{\clif}{Cl}
\DeclareMathOperator{\op}{op}
\DeclareMathOperator{\im}{im}
\DeclareMathOperator{\wt}{wt}
\DeclareMathOperator{\id}{id}
\DeclareMathOperator{\End}{End}

\DeclareMathOperator{\gr}{gr}
\DeclareMathOperator{\spn}{span}
\DeclareMathOperator{\Gr}{Gr}

\DeclarePairedDelimiter\abs{\lvert}{\rvert}%
\DeclarePairedDelimiter\norm{\lVert}{\rVert}%
\makeatletter
\let\oldabs\abs
\def\abs{\@ifstar{\oldabs}{\oldabs*}}
\let\oldnorm\norm
\def\norm{\@ifstar{\oldnorm}{\oldnorm*}}
\makeatother

\usepackage[all]{xy}
\newcommand{\DynkinDiagram}{\xymatrix@M=0pt@=20pt}
\newcommand{\Edge}[2][1]{\ar@[|(1)]@#1{-}[#2]}
\newcommand{\DirectedEdge}[2][1]
{\Edge[#1]{#2}
\ar@[|(1)]@#1{}[#2]
|>>>>{\SelectTips{eu}{}\object@#1{>}}}

\def\tr{\mathop{\mathrm{tr}}}

\title[On the Dolbeault--Dirac operator of quantized
symmetric spaces]{On the Dolbeault--Dirac operator\\
of quantized symmetric spaces} 
\author{Ulrich Kr\"ahmer} 
\address{University of Glasgow, School of
Mathematics and Statistics\\ 15 University Gardens, G12
8QW\\ Glasgow, Scotland}
\email{ulrich.kraehmer@glasgow.ac.uk}

\author{Matthew Tucker-Simmons}
\address{Mathematical Sciences Publishers\\
798 Evans Hall \#3840\\
Berkeley, CA 94720-3840\\
United States}
\email{matt@msp.org}

\thanks{M.T.-S.\ was supported by National Science
Foundation grant DMS-1066368, and U.K.\ was supported
by the EPSRC grant ``Hopf algebroids and operads''.
U.K.\ furthermore thanks INI Cambridge and MSRI
Berkeley for hospitality.}

\begin{document}

\begin{abstract} 
The Dolbeault complex of a quantized
compact Hermitian symmetric space is expressed in terms
of the Koszul complex of a braided symmetric algebra of
Berenstein and Zwicknagl.  This defines
a spectral triple quan\-tizing the
Dolbeault--Dirac operator associated to the canonical
spin\(^c\) structure.  
\end{abstract}

\maketitle

\section{Introduction}
\label{sec:introduction}

The aim of this paper is to connect the noncommutative
geometry of the quantized compact Hermitian symmetric spaces 
(see e.g.\ \cite{DabSit03,DanDab10,DanDabLan08,HecKol04,HecKol06,Kra04,
NesTus05,OBua13,SchWag04}) and the theory of the braided
symmetric and exterior algebras of Berenstein and Zwicknagl 
\cite{BerZwi08,Zwi09}.  

Recall that the compact homogeneous K\"ahler manifolds are the
quotients \(G/P\) of a complex semisimple Lie
group \(G\) by a parabolic subgroup \(P\) (see
\cite[\S 6.4]{BasEas89} or \cite{Wan54}) and that  
the compact Hermitian symmetric spaces 
form a subclass; see \cref{sec:motivation-and-outlook}
and the references therein.
 
The theory of quantum groups leads to a quantization of 
these that is described by a right
coideal subalgebra \(A\) of the quantized function
algebra \(\bbC_q[G]\) \cite{Dij96,NouSug94}.
Furthermore, Connes' spectral triples \cite{Con94,ConMar08}
provide a framework for quantizing the K\"ahler
metric.
Concretely, this requires one to
deform, in addition to the algebra of functions, the 
spinor fields with respect to the canonical
spin\(^c\) structure as well as the Dolbeault--Dirac operator
\(\delbar+\delbar^\ast\) acting on them.  

Generalizing the pioneering work of \cite{DabSit03} 
on \(\bbC \bbP^1\), the existence of such a
spectral triple was established in
\cite{Kra04}, and it was shown that this  
provides a representation of the covariant differential 
calculus over \(A\) (in the sense of Woronowicz) defined and 
studied by Heckenberger and Kolb in \cite{HecKol04,HecKol06}.

However, the implicit nature of the construction of \(\delbar\) 
in \cite{Kra04} prevented both a computation of the spectrum
of \(D\) and a rigorous proof that it defines a nontrivial 
equivariant K-homology class over \(A\).
In the classical case, Parthasarathy's
formula allows one to compute the spectral properties of \(D\) \cite{Par72}; see also \cite{Agr03}.
The present paper is a first step toward a quantum
analogue of this formula.

More precisely, we construct in \cref{dfn:dolbeault-operator} 
an element 
\[
	D=\eth+\eth^* \in \uqg \otimes \clifq
\] 
that we refer to as the \emph{Dolbeault--Dirac operator}. 
This terminology is analogous to Kostant's use of the term ``Dirac
operator'' in \cite{Kos99}: \(D\) is an algebraic object that 
induces the corresponding elliptic first-order
differential operator in the classical setting.    
Here \(\uqg\) is the compact real form of the 
quantized universal
enveloping algebra of the Lie algebra of \(G\) and
\(\clifq\) is a quantized Clifford algebra. 
The main result of our paper is to prove that the 
element \(\eth\) squares to zero so that one has:
\begin{thm}
  \label{thm:main-theorem}
  The Dolbeault--Dirac operator satisfies 
  \(D^2 = \eth \eth^\ast + \eth^\ast \eth\).
\end{thm}
  
The connection to the work of Berenstein and Zwicknagl
arises in the construction of \(\clifq\).  Classically,
the typical fibre of the canonical spin\(^c\) structure
on \(G/P\) is the exterior algebra \(\ext (\fg /
\fp)\), so the classical (complex) Clifford algebra can
be identified with \(\mathrm{End}_\bbC (\ext (\fg / \fp
))\); see \cref{sec:classical-clifford-alg} below.  In
\cref{sec:the-quantum-clifford-algebra}, we construct
the quantized Clifford algebra \(\clifq\) analogously,
using creation and annihilation operators on braided
exterior algebras.  A general theory of these algebras
was developed in \cite{BerZwi08}.  In \cite{Zwi09},
Zwicknagl classified the simple representations of
semisimple Lie algebras whose associated braided
symmetric (and hence exterior) algebras are flat
deformations of the classical ones.  These are
precisely the braided exterior algebras \(\extq(\um)\)
that we consider below.

Furthermore, the quadratic dual of \(\extq(\um)\) is
the braided symmetric algebra \(\symq(\up)\) of the
nilradical of \(\fp\) (recall that \(G/P\) is a
symmetric space if and only if \(\up\) is an abelian
Lie algebra, so that \(S(\up)=U(\up)\); see
\cref{prop:irred-parabolic-conditions} below).
Zwicknagl has embedded these into \(\uqg\)
as what he calls the twisted quantum Schubert cells
\cite{Zwi09}*{\S 5}.  In this way, \(\eth\) becomes
identified with the differential in the Koszul complex
of the Koszul algebra \(\symq(\up)\), and this leads to
\cref{thm:main-theorem}.

The paper is organized as follows:
\cref{sec:notation-and-prelims} contains background
material and notation on semisimple Lie algebras,
parabolic Lie subalgebras, their quantizations
and representations.  All this material is standard
knowledge, except maybe the discussion of the
coboundary structure on the category of Type~1
representations of \(\uqg\) in
\cref{sec:coboundary-structure}.
\cref{sec:quantum-symm-and-ext-algs} recalls the
relevant parts of the work of Berenstein and Zwicknagl.
The subsequent
\cref{sec:quantum-exterior-algebras-are-frobenius} is
devoted to a proof that the quantum exterior algebras
are Frobenius, as this is used in 
the construction of the quantized Clifford
algebra and the proof of \cref{thm:main-theorem} in
\cref{sec:clifford-algs-and-spinors}.  
In \cref{sec:grassmannian} we present some details of these 
constructions in the case when the homogeneous space in question
is the Grassmann manifold \(\Gr(2,4)\).
\cref{sec:motivation} explains the geometric background
as well as the connection of this paper with the
results of \cite{Kra04}.
 
We would like to thank A.~Chirvasitu and V.~Serganova
for helpful discussions regarding aspects of this work.

\section{Notation and preliminaries}
\label{sec:notation-and-prelims}

Here we set notation and recall the facts we will need about quantum enveloping algebras and their attendant infrastructure.
We follow the conventions of \cite{Jan96} for the most part.

\subsection{Simple Lie algebras \(\mathfrak{g}\)}
\label{sec:lie-alg}

Let \(\fg\) be a finite-dimensional complex simple Lie algebra with a fixed Cartan subalgebra \(\fh\) and corresponding root system \(\Delta(\fg) \subseteq \fh^*\).
Let \(\Delta^+(\fg)\) be a choice of positive roots and denote \(\Delta^-(\fg) = - \Delta^+(\fg) = \Delta(\fg) \setminus \Delta^+(\fg)\).
Let \(\Pi = \{\alpha_1, \dots, \alpha_r\}\) be the corresponding set of simple roots, \(\cQ\) the integral root lattice, and \(\qplus \subseteq \cQ\) the cone of nonnegative linear combinations of simple roots.
The Killing form on \(\fg\) is nondegenerate on \(\fh\) and thus induces a symmetric bilinear form \((\cdot,\cdot )\) on \(\fh^*\).  
The Cartan matrix of \(\fg\) is \((a_{ij})\), where
\[ (\alpha_i,\alpha_j) = d_i a_{ij} \quad  \text{ with } \quad  d_i = \frac{(\alpha_i,\alpha_i)}{2}.   \]
Let \(\{\omega_1, \dots, \omega_r\}\) be the corresponding set of fundamental weights, determined by the conditions \((\alpha_i,\omega_j) = \delta_{ij}d_i\).
Then we denote by \(\cP\) the integral weight lattice, and by \(\pplus \subseteq \cP\) the cone of dominant integral weights.

There is a natural partial order on \(\cP\) given by \(\mu \peq \nu\) if \(\nu - \mu \in \qplus\).
We write \(\mu \prec \nu\) if \(\mu \peq \nu\) and \(\mu \neq \nu\).
Since \(\fg\) is simple it has a highest root, which is the highest weight of the adjoint representation.

\subsection{Parabolic subalgebras \(\fp\)}
\label{sec:parabolics}

Given a subset \(\cS \subseteq \Pi\) of simple roots, define two sets of roots by
\[
\Delta(\fl) \eqdef \spn(\cS) \cap \Delta(\fg), \quad 
\Delta(\fu_+) \eqdef \Delta^+(\fg) \setminus \Delta(\fl),
\]
and set
\[
\fl \eqdef \fh \oplus \bigoplus_{\alpha \in \Delta(\fl)} \fg_{\alpha}, \quad 
\fu_{\pm} \eqdef \bigoplus_{\alpha \in \Delta(\fu_+)} \fg_{\pm \alpha}, \quad \text{and} \quad \fp \eqdef \fl \oplus \fu_+.
\]
Then \(\fl\) and \(\fu_\pm\) are Lie subalgebras of \(\fg\), and we have
 \([\fl, \fu_\pm] \subseteq \fu_\pm\).
One calls \(\fp\) the \emph{standard parabolic subalgebra} associated
to \(\cS\).
The subalgebra \(\fl\) is reductive, and is called the \emph{Levi factor} of \(\fp\), while \(\fu_+\) is a nilpotent ideal of \(\fp\), called the \emph{nilradical}.
We refer to the roots in \(\Delta(\up)\) as the \emph{radical roots}.
We also denote the semisimple part of \(\fl\) by \(\fk \eqdef [\fl,\fl]\)
and put
\(\fu \eqdef \fu_+ \oplus \fu_-\).

The adjoint action of \(\fp\) on \(\fg\) descends to an action on \(\fg/\fp\), and the decomposition \(\fg = \fu_- \oplus \fp\) is a splitting as \(\fl\)-modules, so that \(\fg/\fp \cong \fu_-\) as \(\fl\)-modules.
With respect to the Killing form of \(\fg\), both \(\fu_+\) and \(\fu_-\) are isotropic, and we have \(\fl = \fu^\perp\).
Furthermore, the pairing \(\fu_+ \times \fu_- \to \bbC\) coming from the Killing form is nondegenerate, and hence \(\fu_-\) and \(\fu_+\) are mutually dual as \(\fl\)-modules.

\subsection{The cominuscule case}

Throughout this paper we will only deal with a
special type of parabolic subalgebra, which is
characterized by the next result.
  
\begin{prop}
  \label{prop:irred-parabolic-conditions}
  For a standard parabolic subalgebra \(\fp\) of a
complex simple Lie algebra \(\fg\), 
the following conditions are equivalent:
  \begin{enumerate}[(a)]
  \item \(\fg/\fp\) is a simple \(\fp\)-module;
  \item \(\fu_-\) is a simple \(\fl\)-module;
  \item \(\fu_-\) is an abelian Lie algebra;
  \item \(\fu_+\) is a simple \(\fl\)-module;
  \item \(\fu_+\) is an abelian Lie algebra;
  \item \(\fp\) is maximal, i.e.\ \(\cS = \Pi \setminus \{\alpha_t\}\)
    for some \(1 \leq t \leq r\), and moreover \(\alpha_t\) has
    coefficient 1 in the highest root of \(\fg\);
  \item \([\fu,\fu] \subseteq \fl\);
  \item \((\fg,\fl)\) is a symmetric pair, i.e.\ there is an involutive Lie algebra automorphism \(\sigma\) of \(\fg\) such that \(\fl = \fg^\sigma\).
  \end{enumerate}
\end{prop}

\begin{conv}
  \label{conv:cominuscule-assuption}
We assume for the rest of this paper that the 
conditions of \cref{prop:irred-parabolic-conditions} 
are satisfied. We say that a parabolic subalgebra satisfying 
these is of \emph{cominuscule type}.
\end{conv}

Condition (f) allows for a classification of all 
cominuscule parabolics in terms of Cartan data.
The complete list is given in \cref{dynkintable}.
Therein, the simple root \(\alpha_t\) omitted in \(\cS=\Pi
\setminus \{\alpha_t\}\) is indicated 
by crossing out the corresponding node in the Dynkin
diagram of \(\mathfrak{g}\). 
(There is a third diagram missing for type $D_r$, but it can be obtained from 
the second one for type $D_r$ by the nontrivial automorphism of the 
Dynkin diagram.)
See, for example, 
\cite{BasEas89}*{Example 3.1.10} or 
\cite{Kob08}*{\S 7.3} for more information.

\begin{table}[h]
\begin{tabular}{|c|c|}
\hline
& \\
\(A_r\) &
\(\DynkinDiagram{
\bullet \Edge{r} & \bullet \Edge{r} & \,\,\cdots\,\,
\Edge{r} &
\bullet \Edge{r} & \times \Edge{r} & \bullet \Edge{r} &
\,\,\cdots\,\, \Edge{r} & \bullet \Edge{r} &
\bullet}\)\\[2mm]
\(B_r\) &
\(\DynkinDiagram{
\times \Edge{r} & \circ \Edge{r} &
\,\,\cdots\,\, \Edge{r} & \circ \DirectedEdge[2]{r} &
\bullet}\)\\[2mm]
\(C_r\) &
\(\DynkinDiagram{
\bullet \Edge{r} & \bullet \Edge{r} & \,\,\cdots\,\,
\Edge{r} & \bullet & \times
\DirectedEdge[2]{l}}\)\\[2mm]
\(\DynkinDiagram{ \\ D_r \\}\) &
\(\DynkinDiagram{
& & & & \bullet\\
\times \Edge{r} & \bullet \Edge{r} &
\,\,\cdots\,\, \Edge{r} & \bullet \Edge{ur} \Edge{dr}
&\qquad\\
&&&&\bullet
} \quad
\DynkinDiagram{
& & & & \times\\
\bullet \Edge{r} & \bullet \Edge{r} &
\,\,\cdots\,\, \Edge{r} & \bullet \Edge{ur} \Edge{dr}
&\\
&&&&\bullet
}\)\\[2mm]
\(E_6\) &
\(\DynkinDiagram{
\bullet \Edge{r} & \bullet \Edge{r} & \bullet \Edge{r}
\Edge{d}& \bullet \Edge{r} & \times \\
& & \bullet & &}\) \\[2mm]
\(E_7\) &
\(\DynkinDiagram{
\bullet \Edge{r} & \bullet \Edge{r} & \bullet \Edge{r}
\Edge{d}& \bullet \Edge{r} & \bullet \Edge{r}
& \times \\
& & \bullet & & &}\) \\
& \\
\hline
\end{tabular}
\vskip4pt 
\caption{Dynkin diagrams for cominuscule parabolic subalgebras.}
\label{dynkintable}
\end{table}

The next result describes the cominuscule 
situation in more detail:

\begin{lem}
  \label{lem:cominuscule-properties}
  Let \(\fp \subseteq \fg\) be a parabolic subalgebra of cominuscule type.
  Then:
  \begin{enumerate}[(a)]
  \item The radical roots \(\beta \in \Delta(\up)\) are the roots of the form \(\beta = \alpha_t + \xi\), where \(\xi\) is a sum of positive roots of \(\fg\) not involving \(\alpha_t\) (equivalently \((\xi,\omega_t)=0\)).
  \item The weight spaces of the \(\fk\)-modules \(\upm\) are one-dimensional.
  \item The element \(H_{\omega_t}\) spans the center of \(\fl\) and acts as the scalar \(d_t\) in \(\up\).
  \item The highest weight of \(\um\) (respectively lowest weight of \(\up\)) is the restriction to the Cartan subalgebra of \(\fl\) of \(-\alpha_t\) (respectively \(\alpha_t\)).
  \end{enumerate}
\end{lem}

\subsection{The quantized enveloping algebra \(\uqg\)}
\label{sec:enveloping-alg}

Fix a real deformation parameter \(q > 1\).
Let \(\uqg\) be the quantized enveloping algebra in the form denoted \(U_q(\fg, \cP)\) in \cite{Jan96}*{\S 4.5}.
In particular, there are generators \(E_j,F_j,K_\lambda\) for \(j = 1,\dots,r\) and \(\lambda \in \cP\), and we set \(K_j \eqdef K_{\alpha_j}\).
We refer to \cite{Jan96}*{\S 4.3, \S 4.5} for the explicit relations.
There is a Hopf algebra structure on \(\uqg\) such that the \(K_\lambda\) are grouplike and the coproducts of the generators \(E_j,F_j\) are given by
\[
\Delta(E_j) = E_j \otimes 1 + K_j \otimes E_j, \quad \Delta(F_j) = F_j \otimes K_j^{-1} + 1 \otimes F_j.
\]
This uniquely determines the counit and antipode.
We give \(\uqg\) the \(\ast\)-structure known as the compact real form, given
 on the generators by
\[
E_j^\ast = K_jF_j, \quad F_j^\ast = E_j K_j^{-1}, \quad K_\lambda^\ast = K_\lambda.
\]
See \cite{KliSch97}*{\S 6.1.7} for a discussion of real forms.
(The original reference for the classification of real forms is 
\cite{Twi92}. Note that the two sources use opposite coproducts 
(the one in \cite{Twi92} is the same as ours) but the 
\(\ast\)-structures are nevertheless identical.)

We denote by \(\uql\) the Hopf \(\ast\)-subalgebra generated by all \(K_\lambda\) for \(\lambda \in \cP\) together with the \(E_j,F_j\) for \(j \neq t\).
Note that \(K_{\omega_t}\) is central in \(\uql\).

\subsection{Type 1 representations}
\label{sec:representations}

For each \(\lambda \in \pplus\) we denote by \(V_\lambda\) the finite-dimensional irreducible Type~1 representation of \(\uqg\) of highest weight \(\lambda\), and we fix a highest weight vector \(v_\lambda\) in this representation.
Whenever we speak of representations of \(\uqg\) or \(\uql\), we will always implicitly mean representations of Type~1.
By slight abuse of notation we also denote by \(V_\lambda\) the irreducible representation of \(\fg\) with highest weight \(\lambda\).
Following \cite{BerZwi08}, we denote by \(\of\) the category of finite-dimensional Type~1  \(\uqg\)-modules.
For \(V \in \of\), there is a decomposition \(V \cong \bigoplus_{j}V_{\lambda_j}\).

For each \(\lambda \in \pplus\) there is a unique positive-definite Hermitian inner product \(\langle \cdot,\cdot \rangle\) on \(V_\lambda\), conjugate-linear in the first variable, that satisfies \(\langle v_\lambda,v_\lambda \rangle = 1\) and that is invariant under the action of \(\uqg\) in the sense that
\[ \langle x \rhd v, w \rangle = \langle v, x^* \rhd w \rangle  \]
for all \(v,w \in V_\lambda\).
Decomposing an arbitrary \(V \in \of\) into simple submodules induces an inner product on \(V\), and furthermore choosing inner products on \(V,W \in \of\) induces inner products on \(V \otimes W\) and \(W \otimes V\).

\subsection{Braidings}
\label{sec:braidings}

The category \(\of\) is a braided monoidal category \cite{KliSch97}*{Chapter~8}.
For \(V,W \in \of\) we denote by \(\rhat_{VW} \colon V \otimes W \to W \otimes V\) the standard braiding.
The braiding is uniquely determined by the fact that, for weight vectors \(v \in V\) and \(w \in W\), we have
\begin{equation*}
  \label{eq:braiding-weight-vectors}
  \rhat_{VW}(v \otimes w) = q^{(\wt(v),\wt(w))} w \otimes v + \sum_i w_i \otimes v_i,
\end{equation*}
where \(\wt(w_i) \succ \wt(w)\) and \(\wt(v_i) \prec \wt(v)\) for all \(i\).

The braidings \(\rhat_{V,V}\) are diagonalizable, and all eigenvalues are of the form \(\pm q^{a_j}\) for certain \(a_j \in \bbQ\) \cite{KliSch97}*{Corollary~8.23}.
Moreover, the braidings \(\rhat_{V,W}\) are well-behaved with respect to duality in the sense that \((\rhat_{V,W})^{\tr} = \rhat_{V^\ast,W^\ast}\), where the superscript \(\tr\) denotes the dual map (or transpose), and we have identified \((V \otimes W)^\ast \cong W^\ast \otimes V^\ast\) as representations of \(\uqg\).

Similarly, fixing invariant inner products on \(V\) and \(W\) as in \cref{sec:representations}, the braidings are well behaved with respect to adjoints: one can show that \((\rhat_{V,W})^\ast = \rhat_{W,V}\) by noting that both are module maps and that they agree on highest weight vectors.

\subsection{Coboundary structure}
\label{sec:coboundary-structure}

The notion of a coboundary structure on a monoidal category, introduced by Drinfeld in \cite{Dri89}*{\S 3}, is similar to a braiding in the sense that it provides a natural isomorphism \(V \otimes W \overset{\sim}{\longrightarrow} W \otimes V\) for each ordered pair of objects in the category.
While the axioms are less familiar than those for the braided structure, the coboundary structure allows for a cleaner definition of the quantum symmetric and exterior algebras.
\begin{dfn}
  \label{dfn:coboundary-category}
  A \emph{coboundary category} is a monoidal category \(\cC\) together with a natural isomorphism \(\sigma_{VW} : V \otimes W \to W \otimes V\) for all \(V,W \in \cC\) satisfying
  \begin{enumerate}[(i)]
  \item (symmetry axiom) \(\sigma_{WV} \sigma_{VW} = \id_{V \otimes W}\);
  \item (cactus axiom) For all \(U,V,W \in \cC\) the following diagram commutes:
    \begin{equation}
      \label{eq:coboundary-axiom}
      \begin{CD}
        U \otimes V \otimes W @>{\sigma_{UV} \otimes \id}>> V \otimes U \otimes W\\
        @V{\id \otimes \sigma_{VW}}VV @VV{\sigma_{V \otimes U,W}}V\\
        U \otimes W \otimes V @>>{\sigma_{U,W \otimes V}}> W \otimes V \otimes U
      \end{CD}
    \end{equation}
  \end{enumerate}
  Following the terminology of \cite{KamTin09}, we refer to the maps \(\sigma_{VW}\) as \emph{commutors}.
\end{dfn}

Polar decomposition of the braidings leads to a coboundary structure on \(\of\).
In \cite{KamTin09} the authors develop this at the level of suitable completions of \(\uqg\) and its tensor square.
We work here just at the level of representations.
For \(V,W \in \of\), let
\begin{equation}
  \label{eq:def-of-commutors}
  \rhat_{VW} = \sigma_{VW} \bigl( (\rhat_{VW})^* \rhat_{VW} \bigr)^{1/2}
\end{equation}
be the polar decomposition.
Each \(\sigma_{VW}\) is unitary since \(\rhat_{VW}\) is invertible.
We record here the properties of the maps \(\sigma_{VW}\) that we will require:
\begin{prop}
  \label{prop:coboundary-struct-on-uqg-modules}
  The maps \((\sigma_{VW})_{V,W \in \of}\) form a coboundary structure on \(\of\).
  Moreover:
  \begin{enumerate}[(a)]
  \item The diagram
    \[
    \begin{CD}
      (W \otimes V)^* @>{(\sigma_{VW})^{\mathrm{tr}}}>> (V \otimes W)^*\\
      @V{\cong}VV @VV{\cong}V\\
      V^* \otimes W^* @>>{\sigma_{V^* W^*}}> W^* \otimes V^*
    \end{CD}
    \]    
    commutes, where the vertical arrows are the natural isomorphisms \((X \otimes Y)^* \cong Y^* \otimes X^*\) of finite-dimensional \(\uqg\)-modules and \((\sigma_{VW})^{\mathrm{tr}}\) is the transpose of the commutor \(\sigma_{VW}\).
  \item Each \(\sigma_{VW}\) is unitary by construction, and \((\sigma_{VW})^* = \sigma_{WV}\).
  \item In particular, \(\sigma_{VV}\) is a self-adjoint unitary operator with the same eigenspaces as \(\rhat_{VV}\), and \(\sigma_{VV}\) has eigenvalue \(\pm 1\) on an eigenspace according as \(\rhat_{VV}\) has eigenvalue \(\pm q^a\) for some \(a \in \bbQ\).
  \end{enumerate}
\end{prop}

\section{Quantum symmetric and exterior algebras}
\label{sec:quantum-symm-and-ext-algs}

In this section, we discuss some properties of the quantum symmetric and exterior algebras of the representations of \(\uql\) whose classical limits are the  nilradical \(\up\) and its dual \(\um\).
These algebras were introduced by Berenstein and Zwicknagl in \cite{BerZwi08} and studied further in \cite{Zwi09,ChiTuc12}.

Zwicknagl has shown \cite{Zwi09}*{Main Theorem~5.6} that the quantum symmetric algebra \(\symq(\up)\) coincides, as a \(\uql\)-module algebra, with a certain ``twisted'' quantum Schubert cell inside \(\uqg\).
This embedding into the ambient quantized enveloping algebra will give us more precise information about the relations in the quantum symmetric algebra.
In \cref{sec:quantum-exterior-algebras-are-frobenius} we will then transfer this information to the quantum exterior algebra \(\extq(\um)\) by Koszul duality.

As \cite{Zwi09}*{\S 5} is not entirely consistent in its notation, in \cref{sec:berenstein-zwicknagl-quantum-symm-algebras} we go into some detail concerning quantum Schubert cells \(U(w)\) associated to arbitrary elements \(w\) of the Weyl group \(W\) of \(\fg\).
In \cref{sec:twisted-quantum-schubert-cells}  we define a particular element \(\wl \in W\) associated to a cominuscule parabolic and examine \(U(\wl)\), along with its twisted version \(U'(\wl)\).
This turns out to be isomorphic to \(\symq(\up)\) as a \(\uql\)-module algebra, where \(\uql\) acts on \(U'(\wl)\) via the adjoint action of \(\uqg\).

\subsection{Quantization of the nilradical}
\label{sec:quantization-of-nilradical}

Since \(\upm\) are finite-dimensional irreducible representations of \(\fl\), there are corresponding irreducible representations of \(\uql\), which by abuse of notation we denote also by \(\upm\).
As \(\fl\) is reductive rather than semisimple, in addition to the highest weight we must also specify the action of the central generator \(K_{\omega_t}\) of \(\uql\): \(K_{\omega_t}\) acts as \(q_t^{\pm 1} = q^{\pm d_t}\) in \(\upm\), respectively.
We fix a nondegenerate \(\uql\)-invariant pairing \(\langle \cdot,\cdot \rangle: \um \otimes \up \to \bbC\), which defines an isomorphism of \(\uql\)-modules \(\um \cong \up^\ast\).

\subsection{The algebras \(\symq(\up)\) and \(\extq(\up)\)}
\label{sec:berenstein-zwicknagl-quantum-symm-algebras}

Following \cite{BerZwi08}*{Definition 2.7}, the \emph{quantum (or braided) symmetric algebra} of \(\up\) is defined as
\begin{equation}
  \label{eq:quantum-sym-alg-def}
  \symq(\up) = T(\up) / \langle \ker(\sigma_{\up,\up} + \id) \rangle,
\end{equation}
where \(\sigma_{\up,\up}\) is the commutor coming from the coboundary structure on the category of finite-dimensional \(\uql\)-modules.
Similarly, the \emph{quantum exterior algebra} of \(\up\) is given by
\begin{equation}
  \label{eq:quantum-ext-alg-def}
  \extq(\up) = T(\up) / \langle \ker(\sigma_{\up,\up} - \id) \rangle.
\end{equation}
By definition, both \(\symq(\up)\) and \(\extq(\up)\) are quadratic algebras.
We denote their graded components by \(\symq^l(\up)\) and \(\extq^l(\up)\) for \(l \in \zp\), respectively.
As \(\sigma_{\up,\up}\) is a module map, both \(\ker(\sigma_{\up,\up} \pm \id)\) are in fact \(\uql\)-submodules of \(\up \otimes \up\), and hence \(\symq(\up)\) and \(\extq(\up)\) are  \(\uql\)-module algebras.
We denote multiplication in \(\symq(\up)\) and \(\extq(\up)\) by juxtaposition and by the symbol \(\wedge\), respectively.

\begin{rem}
  \label{rem:on-defs-of-sym-and-ext-algebras}
  According to \cref{prop:coboundary-struct-on-uqg-modules}(c), \(\ker(\sigma_{\up,\up} - \id)\) is the span of the eigenspaces of the braiding \(\rhat_{\up,\up}\) for \emph{positive} eigenvalues, i.e., those of the form \(q^a\), while \(\ker(\sigma_{\up,\up}+\id)\) is the span of the eigenspaces for \emph{negative} eigenvalues, i.e., those of the form \(-q^a\).
  Thus, speaking informally, the definitions of the quantum symmetric and exterior algebras become the classical ones when \(q=1\).
\end{rem}

One can define quantum symmetric and exterior algebras analogously for arbitrary finite-dimensional representations of quantized enveloping algebras; the ones defined above are particular cases of this construction.
Zwicknagl has shown that the quantum symmetric and exterior algebras associated to the abelian nilradicals \(\upm\) are much better behaved than those for arbitrary representations.
In particular:

\begin{prop}
  \label{prop:sq-facts-from-zwicknagl}
  With all notation as above, we have:
  \begin{enumerate}[(a)]
  \item The Hilbert series of the quantum symmetric algebra \(\symq(\up)\) is the same as that of the ordinary symmetric algebra \(S(\up)\).
  \item If \(\{ x_1,\dots,x_N \}\) is any basis for \(\up\), then the ordered monomials in these generators form a PBW basis for \(\symq(\up)\).
  \item The quantum symmetric algebra \(\symq(\up)\) is a Koszul algebra.
  \end{enumerate}
\end{prop}

\begin{proof}
  Part (a) is a combination of Theorems~3.14 and~4.23 in \cite{Zwi09}.
  In view of (a), part (b) follows from Proposition~2.28 of \cite{BerZwi08}.
  Part (c) holds because any PBW algebra is Koszul \cite{PolPos05}*{Chapter~4, Theorem~3.1}.
\end{proof}

\subsection{Quadratic duality}
\label{sec:quadratic-duality}

One can make the analogous definitions of \(\symq(\um)\) and \(\extq(\um)\).
Since \(\up\) and \(\um\) are duals of one another, we have by \cref{prop:coboundary-struct-on-uqg-modules}(a) together with the involutivity of the commutors that
\begin{equation}
  \label{eq:quadratic-duality-between-symm-and-ext-algs}
  \symq(\up)^! \cong \extq(\um)
\end{equation}
(see \cite{BerZwi08}*{Proposition~2.11(c)}), where \(A^!\) denotes the \emph{quadratic dual} (or Koszul dual) of a quadratic algebra \(A\).

\begin{rem}
  \label{rem:on-def-of-koszul-dual}
  We note that the definition of the quadratic dual algebra in Chapter~1, Section~2  of \cite{PolPos05} must be modified slightly in our context so that the quadratic dual of a \(\uql\)-module algebra is again a module algebra.
  In particular, we identify \(\um \otimes \um\) with \((\up \otimes \up)^\ast\) via the pairing 
  \begin{equation}
    \label{eq:dual-pairing-of-tensor-squares}
    \langle y \otimes y',x \otimes x' \rangle = \langle y',x \rangle \langle y, x' \rangle
  \end{equation}
  for \(x,x' \in \up\) and \(y,y' \in \um\); it is this pairing which allows us to identify \(\sigma_{\um,\um}\) with the dual map of \(\sigma_{\up,\up}\) as in \cref{prop:coboundary-struct-on-uqg-modules}.
  The usual definition of quadratic dual algebra uses the 
  pairing 
  \[
  \langle y \otimes y', x \otimes x' \rangle = 
  \langle y,x \rangle \langle y',x' \rangle
  \] 
  rather than \eqref{eq:dual-pairing-of-tensor-squares}.
  The consequence of this choice is that \(A^!\) as defined using our convention is the \emph{opposite} algebra of the one defined in \cite{PolPos05}.
\end{rem}

\begin{cor}
  \label{cor:koszulity-of-sq}
  The quantum exterior algebra \(\extq(\um)\) is a Koszul algebra, and has the same Hilbert series as the ordinary exterior algebra \(\ext(\um)\).
\end{cor}

\begin{proof}
  The quadratic dual of a Koszul algebra is again Koszul by Corollary~3.3, Chapter~2 of \cite{PolPos05}.
  The Hilbert series of a Koszul algebra is determined by that of its quadratic dual algebra according to Corollary~2.2, Chapter~2 of \cite{PolPos05}.
  Since the Hilbert series of \(\symq(\up)\) coincides with that of \(\sym(\up)\), the Hilbert series of \(\extq(\um)\) coincides with that of the quadratic dual \(\ext(\um)\) of \(\sym(\up)\).
\end{proof}

\subsection{Linear duality}
\label{sec:linear-duality}

Now we show that \(\extq(\um)\) can be identified with the graded dual of  \(\extq(\up)\).
As in the classical setting, this can be done by embedding \(\extq(\up)\) and \(\extq(\um)\) into the tensor algebras of \(\up\) and \(\um\), respectively.

\begin{dfn}[\cite{BerZwi08}*{Definition 2.1}]
  \label{dfn:antisymmetric-tensors}
  The spaces of \emph{quantum symmetric tensors} \(\symq^n \up \subseteq \up^{\otimes n}\) and  \emph{quantum antisymmetric tensors} \(\extq^n \up \subseteq \up^{\otimes n}\) are defined by 
  \begin{equation}
    \label{eq:antisymmetric-tensors-alternate-definition}
    \symq^n \up \eqdef \bigcap_{j=1}^{n-1} \ker(\sigma_{i} - \id), \quad \extq^n \up \eqdef \bigcap_{j=1}^{n-1} \ker(\sigma_{i} + \id),
  \end{equation}
  respectively, where \(\sigma_i\) is the operator on \(\up^{\otimes n}\) given by \(\sigma_{\up,\up}\) in the \((i,i+1)\) tensor factors and the identity in all others.
  The spaces \(\symq^n \um, \extq^n \um \subseteq \um^{\otimes n}\) are defined analogously.
\end{dfn}

It was shown in \cite{ChiTuc12}*{Proposition~3.2} that one has the
 decompositions of \(\uql\)-modules
\begin{equation}
  \label{eq:decomposition-of-tensor-power}
  \up^{\otimes n} = \extq^n \up \oplus \langle \symq^2 \up \rangle_n, \quad   \um^{\otimes n} = \extq^n \um \oplus \langle \symq^2 \um \rangle_n.
\end{equation}
As \(\symq^2 \upm\) are the spaces of relations in the quantum exterior algebras \(\extq(\upm)\), respectively, the restrictions of the quotient maps from \(\upm^{\otimes n}\) to \(\extq^n(\upm)\) lead to isomorphisms
\begin{equation}
  \label{eq:isos-of-asym-tensors-with-ext-powers}
  \pi^n_\pm : \extq^n \upm \stackrel{\sim}{\longrightarrow} \extq^n(\upm).
\end{equation}

The dual pairing of \(\um\) with \(\up\) extends to a nondegenerate \(\uql\)-invariant pairing \(\langle \cdot,\cdot \rangle : \um^{\otimes n} \otimes \up^{\otimes n} \to \bbC\) generalizing \eqref{eq:dual-pairing-of-tensor-squares}.
With this notation, we have:

\begin{prop}
  \label{prop:dual-of-exterior-algebra}
  For each \(n\), define \(\langle \cdot,\cdot \rangle : \extq^n(\um) \otimes \extq^n(\up) \to \bbC\) by
  \begin{equation}
    \label{eq:dual-pairing-of-exterior-algebras}
    \langle y, x \rangle \eqdef \langle (\pi^n_-)^{-1}(y), (\pi^n_+)^{-1}(x) \rangle,
  \end{equation}
  for \(y \in \extq(\um), x \in \extq(\up)\), where the right-hand side is the pairing of \(\um^{\otimes n}\) with \(\up^{\otimes n}\) introduced above. 
  Then \eqref{eq:dual-pairing-of-exterior-algebras} is nondegenerate and \(\uql\)-invariant, and hence induces an isomorphism \(\extq^n(\um) \cong \extq^n(\up)^\ast\).
\end{prop}

\begin{proof}
  The pairing is \(\uql\)-invariant because it is a composition of \(\uql\)-module maps with a \(\uql\)-invariant pairing.
  In view of \eqref{eq:isos-of-asym-tensors-with-ext-powers}, it suffices to show that \((\extq^n \um)^\perp = \langle \symq^2 \up \rangle_n\), where \(\perp\) denotes the annihilator with respect to the dual pairing of \(\um^{\otimes n}\) with \(\up^{\otimes n}\).
  The cases when \(n=0,1\) are clear.
  
  For \(n=2\), as \(\sigma_{\up,\up}\) is involutive we have
  \[
  \extq^2\up = \ker(\sigma_{\up,\up} + \id) = \im(\sigma_{\up,\up} - \id).
  \]
  Then using \cref{prop:coboundary-struct-on-uqg-modules}(a) we obtain
  \[
  (\extq^2\up)^\perp = \im(\sigma_{\up,\up} - \id)^\perp = \ker(\sigma_{\um,\um} - \id) = \symq^2 \up.
  \]

  For \(n > 2\), note that
  \[
  \ker(\sigma_i + \id) = \up^{\otimes i-1} \otimes \extq^2\up \otimes \up^{\otimes n-i-1},
  \]
  and hence
  \[
  \ker(\sigma_i + \id)^\perp = \um^{\otimes i-1} \otimes \symq^2\um \otimes \um^{\otimes n-i-1}.
  \]
  Finally, this gives
  \[
  (\extq^n\up)^\perp = \biggl( \bigcap_{i=1}^n \ker(\sigma_i + \id) \biggr)^\perp = \sum_{i=1}^n \ker(\sigma_i + \id)^\perp =  \langle \symq^2 \up \rangle_n.
  \]
  This completes the proof.
\end{proof}

\subsection{The algebra \(U(w)\)}
\label{sec:quantum-schubert-cell}

We now recall some definitions and results concerning Weyl groups and parabolic subgroups.
Let \(W\) be the Weyl group of \(\fg\), and let \(s_i \in W\) be the reflection corresponding to the simple root \(\alpha_i\).

\begin{dfn}
  \label{dfn:Phi-w}
  For any word \(w \in W\), define \(\Phi(w) \subseteq \Delta^+(\fg)\) by
  \begin{equation}
    \label{eq:Phi-w-def}
    \Phi(w) \eqdef \Delta^+(\fg) \cap w(\Delta^-(\fg)),
  \end{equation}
  so that \(\Phi(w)\) is the set of positive roots \(\beta\) such that \(w^{-1} (\beta)\) is negative.  
\end{dfn}

The following result can be found, for example, in \S 1.7 of \cite{Hum90}.
Note, however, that our \(\Phi(w)\) is \(\Pi(w^{-1})\) in Humphreys' notation.

\begin{lem}
  \label{lem:Phi-w-recipe}
  Given any reduced expression \(w = s_{i_1} \dots s_{i_m}\) for \(w\), the sequence
  \begin{equation}
    \label{eq:Phi-w-recipe}
    s_{i_1} \dots s_{i_{k-1}}(\alpha_{i_k}), \quad 1 \leq k \leq m
  \end{equation}
  consists of \(m\) distinct positive roots, and exhausts \(\Phi(w)\).
  In particular, \(\ell(w) = \abs{\Phi(w)}\), where \(\ell\) is the word-length function of the Weyl group with respect to the generators \(\{ s_i \}\).
\end{lem}

\begin{dfn}
  \label{dfn:quantum-schubert-cells}
  Retaining the notation from \cref{lem:Phi-w-recipe}, the \emph{quantum Schubert cell} \(U(w)\) is defined to be the subalgebra of \(\uqg\) generated by the elements
  \begin{equation}
    \label{eq:quantum-schubert-cell-generators}
    T_{i_1} \dots T_{i_{k-1}}(E_{i_k}), \quad 1 \le k \le m,
  \end{equation}
  where the \(T_i\) are the automorphisms of \(\uqg\) given by the action of the braid group \(B_\fg\); we use the conventions of \cite{Jan96}*{Chapter~8} here.
\end{dfn}

\begin{rem}
  \label{rem:on-the-quantum-schubert-cells}
  Although the generators \eqref{eq:quantum-schubert-cell-generators} depend on the choice of reduced expression for \(w\), the algebra \(U(w)\) is independent of this choice (see \S 9.3 of \cite{ConPro93}, for example).
  The quantum Schubert cells were introduced in \cite{ConKacPro95} and are hence often called De Concini--Kac--Procesi algebras.
  Zwicknagl shows in \S 5 of \cite{Zwi09} that \(U(w)\) can also be defined as 
  \begin{equation}
    \label{eq:alternate-def-of-quantum-schubert-cell}
    U(w) = T_w(\uqbm) \cap \uqnp
  \end{equation}
  (although note that the \(w^{-1}\) appearing in Definition~5.2 of \cite{Zwi09} must be changed to \(w\), and Theorem~5.21(a) should say \(U'(w)\) rather than \(U(w)\)).
\end{rem}

\subsection{The algebra \(U'(w)\)}
\label{sec:twisted-quantum-schubert-cells}

Now we recall Zwicknagl's definition of the twisted quantum Schubert cell that we alluded to above.
\begin{dfn}
  \label{dfn:twisted-quantum-schubert-cell}
  For an arbitrary element \(w \in W\), the \emph{twisted quantum Schubert cell} \(U'(w)\) is defined by
  \begin{equation}
    \label{eq:twisted-quantum-schubert-cell-dfn}
    U'(w) = T_{\wz w^{-1}} U(w).
  \end{equation}
  (Recall from \S 8.18 of \cite{Jan96} that \(T_{uv} = T_u T_v \) if \(\ell(uv) = \ell(u) + \ell(v)\), but this is not the case in general.)
\end{dfn}

We now examine a particular case of this construction.
Let \(\fl\) be the Levi factor of a cominuscule
parabolic subalgebra \(\fp \subseteq \fg\), and let
\(\Wl\) be the \emph{parabolic subgroup} of the Weyl
group \(W\) of \(\fg\), generated by the simple
reflections \(s_j\) for \(j \neq t\).  We refer to
\(\Wl\) as the Weyl group of \(\fl\), although properly
speaking it is the Weyl group of 
\(\fk\) (see \cite{Hum90}*{\S 5.5}).  Let \(\wzl\)
be the longest word of \(\Wl\), and define \(\wl = \wzl
\wz\), where \(\wz\) is the longest word of \(W\); we
refer to \(\wl\) as the \emph{parabolic element}.

\begin{lem}
  \label{lem:parabolic-element-properties}
  With the parabolic element \(\wl\) defined as above, we have:
  \begin{equation}
    \label{eq:factorization-of-w0}
    \wz = \wzl \wl,
  \end{equation}
  and moreover
  \begin{enumerate}[(a)]
  \item \(\ell(\wl) + \ell(\wzl) = \ell(\wz)\), where \(\ell\) is the word-length function on \(W\).
  \item The set \(\Delta(\up)\) is invariant under the action of \(\Wl\).
  \item The set of positive roots of \(\fg\) mapped by \(\wl^{-1}\) to negative roots is precisely the set of radical roots, i.e.\ \(\Phi(\wl) = \Delta(\up)\).
  \end{enumerate}
\end{lem}

\begin{proof}
  Since \(\wzl\) is the longest word of a Weyl group, it is involutive, and hence \eqref{eq:factorization-of-w0} follows immediately from the definition of \(\wl\).
  Part (a) is Equation (2) in \S 1.8 of \cite{Hum90}.
  
  Part (b) follows from the fact that \(\up\) is invariant under the adjoint action of \(\fl\).

  To prove part (c), let \(\beta \in \Delta^+(\fg)\).
  By definition, we have \(\wl^{-1}(\beta) = \wz (\wzl(\beta))\).
  On the one hand, if \(\beta \in \Delta^+(\fl)\), then \(\wzl(\beta)\) is a negative root of \(\fl\) since \(\wzl\) is the longest word of \(\Wl\), and hence \(\wz(\wzl(\beta))\) is a positive root of \(\fg\).
  On the other hand, if \(\beta \in \Delta(\up)\), then \(\wzl(\beta) \in \Delta(\up)\) by part (b), so \(\wz(\wzl(\beta))\) is a negative root of \(\fg\).
  Hence \(\Phi(\wl) = \Delta(\up)\).
\end{proof}

\begin{conv}
  \label{conv:reduced-expression-of-w0}
  We fix a reduced expression for the longest word
  \begin{equation}
    \label{eq:reduced-expression-of-w0}
    \wz = s_{j_1}\dots s_{j_M} s_{j_{M+1}} \dots s_{j_{M+N}}
  \end{equation}
  which is compatible with \eqref{eq:factorization-of-w0} in the sense that \(\wzl = s_{j_1}\dots s_{j_M}\) and \(\wl = s_{j_{M+1}} \dots s_{j_{M+N}}\) for some \(M \geq 0\) and \(N \geq 1\).
  (The case \(M=0\) arises when \(\fp\) is a Borel subalgebra of \(\fs\fl_2\).)
  The expressions for \(\wzl\) and \(\wl\) are both reduced since they are intervals in the reduced expression \eqref{eq:reduced-expression-of-w0}.
\end{conv}

\begin{dfn}
  \label{dfn:radical-roots-sequence}
  For \(1 \le k \le N\), define a root \(\xi_k\) of \(\fg\) by
  \begin{equation}
    \label{eq:radical-roots-sequence}
    \xi_k \eqdef s_{j_1} \dots s_{j_M}s_{j_{M+1}} \dots s_{j_{M+k-1}}(\alpha_{j_{M+k}}) = \wzl s_{j_{M+1}} \dots s_{j_{M+k-1}}(\alpha_{j_{M+k}}).
  \end{equation}
\end{dfn}

The last result we need about the parabolic element is:
\begin{lem}
  \label{lem:alternate-recipe-for-radical-roots}
  The set \(\{ \xi_1, \dots, \xi_N \}\) is precisely the set \(\Delta(\up)\) of radical roots.
\end{lem}

\begin{proof}
  Since \(\wl = s_{j_{M+1}} \dots s_{j_{M+N}}\), according to \cref{lem:Phi-w-recipe} the roots of the form \(s_{j_{M+1}} \dots s_{j_{M+k-1}}(\alpha_{j_{M+k}})\) for \(1 \leq k \leq N\) exhaust \(\Phi(\wl)\).
  By \cref{lem:parabolic-element-properties}(c) this set is exactly \(\Delta(\up)\).
  By \cref{lem:parabolic-element-properties}(b) the element \(\wzl\) permutes \(\Delta(\up)\), which completes the proof.
\end{proof}

Using the reduced expression \eqref{eq:reduced-expression-of-w0} for \(\wz\) and the action of the braid group on \(\uqg\), the \emph{quantum root vectors} of \(\uqg\) are defined as follows.
For \(\beta \in \Delta^+(\fg)\), the element \(E_\beta\) is given by
\begin{equation}
  \label{eq:quantum-root-vectors-def}
  E_\beta \eqdef T_{j_1} \dots T_{j_{k-1}}(E_{j_k}),
\end{equation}
where \(k\) is the unique index for which \(\beta = s_{j_1} \dots s_{j_{k-1}}(\alpha_{j_k})\).
We have \(E_{\alpha_j} = E_j\) for the simple roots \(\alpha_j\) by \cite{Jan96}*{Lemma~8.19}.
Similarly, \(F_\beta\) is defined by replacing \(E_{j_k}\) with \(F_{j_k}\) in \eqref{eq:quantum-root-vectors-def}.
The quantum root vectors \(\{E_\beta\}\) give rise to a PBW basis of \(\uqnp\), and similarly the \(\{F_\beta\}\) give a PBW basis of \(\uqnm\) \cite{Jan96}*{Theorem~8.24}.
The quantum root vectors depend on the choice of reduced decomposition of \(\wz\); we always use \eqref{eq:reduced-expression-of-w0}.

\begin{rem}
  \label{rem:gens-of-quantum-schub-cell-not-quantum-root-vectors}
  Using the reduced expression \(\wl = s_{j_{M+1}} \dots s_{j_{M+N}}\), the quantum Schubert cell \(U(\wl)\) is, by definition, generated by the elements
  \begin{equation}
    \label{eq:gens-of-quantum-schubert-cell}
    X_k \eqdef T_{j_{M+1}} \dots T_{j_{M+k-1}}(E_{j_{M+k}}), \quad 1 \leq k \leq N.
  \end{equation}
  Note that these generators are \emph{not} the quantum root vectors associated to the roots \(\{ \xi_k \}\).
  However, according to the definition \eqref{eq:quantum-root-vectors-def} of the quantum root vectors and the definition \eqref{eq:radical-roots-sequence} of the \(\xi_k\)'s, we do have
  \begin{equation}
    \label{eq:quantum-root-vectors-for-radical-roots}
    E_{\xi_k} = T_{j_1} \dots T_{j_M}T_{j_{M+1}} \dots T_{j_{M+k-1}}(E_{j_{M+k}}) = T_{\wzl}(X_k).
  \end{equation}
  Thus the twisted quantum Schubert cell \(U'(\wl) \eqdef T_{\wzl} U(\wl)\) is generated by the quantum root vectors \(\{ E_{\xi_1}, \dots, E_{\xi_N} \}\).
\end{rem}

Since \(\uqg\) is a Hopf algebra, it is a left module algebra over itself via the \emph{adjoint action} \(\ad(a)x = a_{(1)}xS(a_{(2)})\) (using Sweedler's notation for the coproduct of \(a\)).
The reason for introducing the twisted quantum Schubert cell is:

\begin{prop}[\cite{Zwi09}*{Theorem~5.6}]
  \label{prop:zwicknagl-main-theorem-5-6}
  With all notation as above, we have:
  \begin{enumerate}[(a)]
  \item The twisted quantum Schubert cell \(U'(\wl)\) is a quadratic algebra with the generators \(\{ E_{\xi_j} \}\) in degree one and relations of the form
    \begin{equation}
      \label{eq:commutation-rels-for-quantum-root-vectors}
      E_{\xi_l} E_{\xi_k} - q^{-(\xi_k,\xi_l)}E_{\xi_k} E_{\xi_l} = \sum_{k < i \leq j < l} c^{ij}_{kl}E_{\xi_i} E_{\xi_j}
    \end{equation}
    for \(k < l\), where \(c^{ij}_{kl} \in \bbQ[q^{\pm 1}]\).
  \item \(U'(\wl)\) is invariant under the adjoint action of \(\uql\).
  \item \(U'(\wl)\) is isomorphic to the quantum symmetric algebra \(\symq(\up)\) as a graded \(\uql\)-module algebra.
  \end{enumerate}
\end{prop}

\begin{rem}
  \label{rem:on-zwicknagls-main-theorem-5-6}
  Note that \(U(\wl)\) is in general not invariant under \(\uql\), which is the reason for introducing \(U'(\wl)\).  However, since \(U'(\wl)\) is obtained from \(U(\wl)\) by acting with the automorphism \(T_{\wzl}\), the two algebras are isomorphic.
  The fact that \(U'(\wl)\) is quadratic follows from the commutation relations \cite{ConPro93}*{Theorem~9.3} for the quantum root vectors:
  for any \(k < l\), the \(q\)-commutator \(E_{\xi_l} E_{\xi_k} - q^{-(\xi_k,\xi_l)}E_{\xi_k} E_{\xi_l}\) can be expressed as a linear combination of products \(E_{\xi_{j_1}} \dots E_{\xi_{j_m}}\) with \(k < j_1 \leq \dots \leq j_m < l\) and \(\sum_{i=1}^m \xi_{j_i} = \xi_k + \xi_l\).
  In the cominuscule situation all \(\xi_j\) contain \(\alpha_t\) with coefficient one, so we must have \(m=2\).
  Then the PBW theorem for \(\uqg\) shows that these are all of the relations between the generators for \(U'(\wl)\). 
\end{rem}

\section{The quantum exterior algebras are Frobenius algebras}
\label{sec:quantum-exterior-algebras-are-frobenius}

This section contains the 
technical proof that the quantum exterior algebra 
\(\extq(\um)\) is a Frobenius algebra.
We prove this by introducing a one-generated filtration of \(\symq(\up)\) by a graded ordered semigroup in the sense of \cite{PolPos05}*{Chapter~4, \S 7}.
Using the theory developed therein, we transfer this filtration to the quadratic dual algebra \(\extq(\um)\).
It is straightforward to prove that the associated graded algebra is Frobenius, and then we lift the result to \(\extq(\um)\) itself.

\subsection{A filtration of the quantum symmetric algebra}
\label{sec:filtratation-of-quantum-symmetric-algebra}

We begin by introducing a distinguished generating set for \(\symq(\up)\).

\begin{conv}
  \label{conv:generators-of-quantum-symmetric-algebra}
  We fix a weight basis \(\{ x_j \}_{j=1}^N\) for the \(\uql\)-module \(\up\) such that \(x_j\) corresponds to the quantum root vector \(E_{\xi_j}\) under the isomorphism \(\symq(\up) \cong U'(\wl)\) from \cref{prop:zwicknagl-main-theorem-5-6}(c).
  Thus from \eqref{eq:commutation-rels-for-quantum-root-vectors} we get the relations 
  \begin{equation}
    \label{eq:relations-for-quantum-symmetric-algebra}
    x_l x_k - q^{-(\xi_k,\xi_l)}x_k x_l = \sum_{k < i \leq j < l} c^{ij}_{kl}x_i x_j, \quad k < l
  \end{equation}
  among the generators of the quantum symmetric algebra.
  For \(\bk = (k_1, \dots, k_N) \in \zp^N\), we define the monomial
  \begin{equation}
    \label{eq:def-of-PBW-monomials}
    x_\bk \eqdef x_1^{k_1} \dots x_N^{k_N}.
  \end{equation}
  Then according to \cref{prop:sq-facts-from-zwicknagl}(b), the set \(\{ x_\bk \}\) is a PBW basis for \(\symq(\up)\).
\end{conv}

\begin{dfn}
  \label{dfn:filtration-of-quantum-symmetric-algebra}
  Let \(\Gamma = \zp^N\), and denote by \(\delta_j\) the element with a \(1\) in the \(j^{\text{th}}\) position and zeros elsewhere.
  Define a semigroup homomorphism \(g : \Gamma \rightarrow \zp\) by
  \[
  g(k_1,\dots, k_N) \eqdef \sum_j k_j.
  \]
  For \(l \in \zp\) denote \(\Gamma_l = g^{-1}(l)\).
  We give each \(\Gamma_l\) the lexicographic order, i.e.\ we say that \((k_1,\dots,k_N) < (k_1',\dots,k_N')\) if there is an index \(j\) such that \(k_i = k_i'\) for \(i < j\) and \(k_j < k_j'\).
  (Note that \(\delta_1 > \dots > \delta_N\) in this ordering.)
  For \(\bk \in \Gamma_l\), we define a subspace \(\cF_\bk = \cF_\bk \symq(\up) \subseteq \symq^l(\up)\) by
  \begin{equation}
    \label{eq:filtration-of-quantum-symmetric-algebra}
    \cF_\bk = \cF_\bk \symq(\up) \eqdef \spn \{ x_{\bk'} \mid \bk' \in \Gamma_l \text{ and } \bk' \leq \bk\} \subseteq \symq^l(\up).
  \end{equation}
\end{dfn}

\begin{lem}
  \label{lem:filtration-is-one-generated}
  The set \(\{ \cF_{\bk} \mid \bk \in \Gamma \}\) is a \emph{\(\Gamma\)-valued filtration} of \(\symq(\up)\), i.e.\ the following hold:
  \begin{enumerate}[(a)]
  \item \(\cF_{\bj} \subseteq \cF_{\bk}\) when \(\bj \leq \bk \in \Gamma_l\).
  \item \(\cF_{(l,0,\dots,0)} = S^l_q(\up)\) (note that \((l,0,\dots,0)\) is the maximal element in \(\Gamma_l\)).
  \item \(\cF_{\bk} \cF_{\bk'} \subseteq \cF_{\bk + \bk'}\).
  \end{enumerate}
  Moreover, the filtration is \emph{one-generated}, i.e.\ for every \(\bk \in \Gamma_l\) we have
  \begin{equation}
    \label{eq:one-generation-of-filtration}
    \cF_{\bk} = \sum_{\delta_{i_1} + \dots + \delta_{i_l} \leq \bk}\cF_{\delta_{i_1}} \dots \cF_{\delta_{i_l}}.
  \end{equation}
\end{lem}

\begin{proof}
  Parts (a) and (b) follow immediately from the definitions.
  Part (c) follows from the commutation relations \eqref{eq:relations-for-quantum-symmetric-algebra}.
  Finally, the filtration is one-generated because the \(x_j\) generate \(\symq(\up)\).
\end{proof}

The associated \(\Gamma\)-graded algebra \(\grf \symq(\up)\) is also \(\zp\)-graded via the homomorphism \(g : \Gamma \to \zp\).
Its relations are particularly simple:

\begin{prop}
  \label{prop:associated-graded-of-quantum-symmetric-algebra}
  The associated \(\Gamma\)-graded algebra \(\grf \symq(\up)\) is generated by the elements \(\{ \ox_j \}_{j=1}^N\) subject to the defining relations
  \begin{equation}
    \label{eq:commutation-rels-in-associated-graded}
    \ox_l \ox_k - q^{-(\xi_k,\xi_l)}\ox_k \ox_l = 0, \quad k < l,    
  \end{equation}
  where \(\ox_k\) is the image of \(x_k\) in \(\cF_{\delta_k} / \cF_{\delta_{k+1}}\).
\end{prop}

\begin{proof}
  The elements \(\{ \ox_k \}\) generate the associated graded algebra because the filtration is one-generated.
  In the relations \eqref{eq:relations-for-quantum-symmetric-algebra}, the terms on the right-hand side have lower filtration degree than those on the left-hand side, and thus vanish in the associated graded algebra.
  The relations \eqref{eq:commutation-rels-in-associated-graded} follow.
  These are the only relations in \(\grf \symq(\up)\) because the \(x_\bk\) for \(\bk \in \Gamma\) form a PBW basis of \(\symq(\up)\).
\end{proof}

\subsection{The dual filtration of the quantum exterior algebra}
\label{sec:dual-filtration-of-quantum-exterior-algebra}

As \(\extq(\um)\) is the quadratic dual of \(\symq(\up)\), it comes with a \emph{dual filtration:}

\begin{dfn}
  \label{dfn:filtration-of-quantum-exterior-algebra}
  Let \(\gc \eqdef \zp^N\) denote the same semigroup as \(\Gamma\), but with the opposite ordering on each fiber \(\gc_l \eqdef g^{-1}(l)\), so that \(\delta_1 < \dots < \delta_N\).
  Let \(\{ y_j \}_{j=1}^N\) be the basis for \(\um\) dual to \(\{ x_j \}_{j=1}^N\), i.e.
  \begin{equation}
    \langle y_i,x_j \rangle = \delta_{ij},
  \end{equation}
  where \(\langle \cdot,\cdot \rangle\) is the pairing from \cref{sec:quantization-of-nilradical}.
  We give \(\extq(\um) = \symq(\up)^!\) the one-generated \(\gc\)-valued filtration \(\Fc\) determined by
  \begin{equation}
    \label{eq:def-of-filtration-of-quantum-exterior-algebra}
    \Fc_{\delta_k} \eqdef \spn \{ y_j \mid j \leq k\}, \quad 1 \leq k \leq N.
  \end{equation}
  Then for arbitrary \(\bk \in \gc_l\) the subspace \(\Fc_{\bk} = \Fc_{\bk} \extq(\um) \subseteq \extq^l(\um)\) is defined by the analogous formula to \eqref{eq:one-generation-of-filtration}, keeping in mind that we use the ordering of \(\gc\).
\end{dfn}

\begin{prop}
  \label{prop:associated-graded-of-quantum-exterior-algebra}
  The associated \(\gc\)-graded algebra \(\grfc \extq(\um)\) is generated by the elements \(\{ \oy_j \}_{j=1}^N\), subject to the defining relations
  \begin{equation}
    \label{eq:commutation-rels-in-associated-graded-exterior-algebra}
    \oy_l \wedge \oy_k +  q^{-(\xi_k,\xi_l)} \oy_k \wedge \oy_l = 0, \quad k \leq l,
  \end{equation}
  where \(\oy_k\) is the image of \(y_k\) in \(\Fc_{\delta_k} / \Fc_{\delta_{k-1}}\).
\end{prop}

\begin{proof}
  The elements \(\{ \oy_k \}\) generate the associated graded algebra because the filtration \(\Fc\) is one-generated.
  From \cref{prop:associated-graded-of-quantum-symmetric-algebra} we can see that \(\grf \symq(\up)\) is PBW, and hence Koszul.
  Then according to Corollary~7.3 in Chapter~4 of \cite{PolPos05}, \(\grfc \extq(\um)\) is Koszul, and we have
  \[
  \grfc \extq(\um) = \grfc \bigl( \symq(\up)^! \bigr) \cong \bigl( \grf \symq(\up) \bigr)^!.
  \]
  It is straightforward to show that the relations dual to \eqref{eq:commutation-rels-in-associated-graded} are exactly \eqref{eq:commutation-rels-in-associated-graded-exterior-algebra}.
\end{proof}

\begin{rem}
  \label{rem:heckenberger-kolb-already-proved-this}
  \cref{prop:associated-graded-of-quantum-exterior-algebra} was proved by different methods in \cite{HecKol06}, Proposition~3.7.
\end{rem}

\begin{dfn}
  \label{dfn:basis-of-quantum-exterior-algebra}
  For any subset \(J \subseteq \{ 1, \dots, N \}\) with \(\abs{J} = l\), we define elements \(y_J \in \extq^l(\um)\) and \(x_J \in \extq(\up)\) by
  \begin{equation}
    \label{eq:basis-of-quantum-exterior-algebra}
    y_{J} \eqdef y_{j_1} \wedge \dots \wedge y_{j_l}, \quad x_{J} \eqdef x_{j_1} \wedge \dots \wedge x_{j_l}
  \end{equation}
  where \(J = \{ j_1, \dots, j_l \}\) and \(j_1 < \dots < j_l\).
  We denote \([N] \eqdef \{ 1,\dots,N \}\).
\end{dfn}

\begin{cor}
  \label{cor:basis-for-quantum-exterior-power}
  The elements \(y_J\) with \(\abs{J} = l\) form a basis for \(\extq^l(\um)\).
\end{cor}

\begin{proof}
  It follows from \eqref{eq:commutation-rels-in-associated-graded-exterior-algebra} that the elements \(y_J\) with \(\abs{J} = l\) span \(\extq^l(\um)\).
  \cref{cor:koszulity-of-sq} implies that the dimension of \(\extq^l(\um)\) is \(\binom{N}{l}\), so these \(y_J\) are linearly independent, and hence form a basis.
\end{proof}

\subsection{The Frobenius property}
\label{sec:frobenius-property-of-quantum-exterior-algebra}

\cref{prop:associated-graded-of-quantum-exterior-algebra} implies that \(\grfc \extq(\um)\) is a Frobenius algebra with Frobenius functional given by projection onto \(\oy_{[N]}\).

\newcommand{\lmax}{l_{\mathrm{max}}}

\begin{lem}
  \label{lem:frobenius-lifts}
  If \(A\) is a finite-dimensional algebra with a \(\Gamma\)-valued filtration \(\cF\) such that \(\grf A\) is a Frobenius algebra, then \(A\) is also a Frobenius algebra.
\end{lem}

\begin{proof}
  See Theorem~2 of \cite{Bon67}, where the case \(\Gamma = \zp\) is treated.
  The proof for an arbitrary graded ordered semigroup is a straightforward generalization.
\end{proof}

\begin{prop}
  \label{prop:frobenius-for-quantum-exterior-algebras}
  \begin{enumerate}[(a)]
  \item The quantum exterior algebras \(\extq(\up)\) and \(\extq(\um)\) are Frobenius algebras.
  \item Frobenius functionals for \(\extq(\up)\) and \(\extq(\um)\) are given by projection onto \(x_{[N]}\) and projection onto \(y_{[N]}\), respectively.
  \end{enumerate}
\end{prop}

\begin{proof}
  From the relations \eqref{eq:commutation-rels-in-associated-graded-exterior-algebra} we see that projection onto \(\oy_{[N]}\) is a Frobenius functional for \(\grfc \extq(\um)\).
  Then \cref{lem:frobenius-lifts} implies that \(\extq(\um)\) itself is a Frobenius algebra.
  The analogous arguments apply to \(\extq(\up)\).
\end{proof}

\begin{rem}
  \label{rem:frobenius-algebra-but-not-in-uql-mod}
  Although the quantum exterior algebra \(\extq(\um)\) is both a Frobenius algebra and a \(\uql\)-module algebra, it is not a Frobenius algebra in the category of \(\uql\)-modules.
  The reason is that the Frobenius functional is not equivariant for the action of \(K_{\omega_t}\).  
  Indeed, \(K_{\omega_t}\) acts as the scalar \(q^{-N d_t}\) on \(\bbC y_{[N]}\), but as the identity on \(\bbC\).
  The analogous remark applies to \(\extq(\up)\) as well.
\end{rem}

\section{Quantum Clifford algebras and the 
Dolbeault--Dirac operator}
\label{sec:clifford-algs-and-spinors}

We now introduce the quantum Clifford algebra 
\(\clifq\) via its spinor representation.
In \cref{sec:classical-clifford-alg} we recall the 
realization of the Clifford algebra of a hyperbolic quadratic 
space \(V \oplus V^\ast\) by creation and annihilation operators.
This amounts to factorizing \(\End_\bbC(\ext(V))\) as a product 
of two subalgebras isomorphic to \(\ext (V)\) and
\(\ext (V^\ast)\). 
In \cref{sec:the-quantum-clifford-algebra} we generalize this factorization to the quantum setting.
Finally, we use 
\(\clifq\) to define an algebraic analogue 
\(D \in \uqg \otimes 
\clifq\) of the Dolbeault--Dirac operator. The key point
is that \(\uqg \otimes \clifq\) is an
extension of \(\symq(\up)^\mathrm{op} \otimes
\extq(\um)\) to a
\(\ast\)-algebra, and \(D\) is obtained as the sum of the Koszul
boundary map of \(\symq(\up)\) plus its formal adjoint. 

\subsection{The classical Clifford algebra}
\label{sec:classical-clifford-alg}

Let \(V\) be a finite-dimensional complex vector space with dual space \(V^\ast\).
Then \(V \oplus V^\ast\) carries the canonical symmetric bilinear form, which gives rise to the Clifford algebra \(\clif(V \oplus V^\ast)\).
The exterior algebra \(\ext(V)\) can be used as a model for the space \(\cS\) of spinors.
More precisely, as \(V\) and \(V^\ast\) are isotropic, the exterior algebras \(\ext(V)\) and \(\ext(V^\ast)\) embed as subalgebras into the Clifford algebra, and the multiplication map of the Clifford algebra is an isomorphism of vector spaces
\begin{equation}
  \label{eq:multiplication-isomorphism}
  \ext(V^\ast) \otimes \ext(V) \overset{\cong}{\longrightarrow} \clif(V \oplus V^\ast).
\end{equation}
The representation of \(\clif(V \oplus V^\ast)\) on \(\cS = \ext(V)\) restricts to the regular representation of \(\ext(V)\) and to the dual of the regular representation of \(\ext(V^\ast)\), respectively.
In this way, we obtain a factorization
\begin{equation}
\gamma : \ext(V^\ast) \otimes \ext(V) \overset{\cong}{\longrightarrow} \End_\bbC(\cS)\label{eq:factorization-of-clifford-algebra}
\end{equation}
of the endomorphism algebra of \(\cS\) into the product of two subalgebras.
For further details, see \cite{Bas74,Che54}.

\subsection{The quantum Clifford algebra}
\label{sec:the-quantum-clifford-algebra}

Taking \(V = \up\) and \(V^\ast = \um\), we now construct an analogue of the isomorphism \eqref{eq:factorization-of-clifford-algebra}, replacing \(\ext(\upm)\) with \(\extq(\upm)\), respectively.
In fact, we will carry out this construction \(\uql\)-equivariantly, so we 
must keep track of left and right duals.

Let \(\gamma_+\) denote the left regular representation of \(\extq(\up)\) on itself, so that
\[
\gamma_+(x)z = x \wedge z
\]
for \(x,z \in \extq(\up)\).
The resulting algebra map \(\gamma_+ : \extq(\up) \to \End_\bbC(\extq(\up))\) is \(\uql\)-equivariant because \(\extq(\up)\) is a \(\uql\)-module algebra.
The operators \(\gamma_+(x)\) are the \emph{quantum creation operators}.

Now we define the \emph{quantum annihilation operators} as follows.
In \cref{prop:dual-of-exterior-algebra} we identified \(\extq(\um) \cong \extq(\up)^\ast\), and hence we have \(\extq(\up) \cong {^\ast}\extq(\um)\) as left \(\uql\)-modules.
Thus \(\extq(\up)\) is a left \(\extq(\um)\)-module, identified with the dual of the right regular representation.
We denote this action of \(\extq(\um)\) on \(\extq(\up)\) by \(\gamma_-\).
It is given explicitly by
\[
\langle w, \gamma_-(y) x \rangle = \langle w \wedge y, x \rangle
\]  
for \(x \in \extq(\up)\) and \(w,y \in \extq(\um)\).
The algebra map \(\gamma_- : \extq(\um) \to \End_\bbC(\extq(\up))\) is \(\uql\)-equivariant because the pairing is invariant and because \(\extq(\um)\) is a \(\uql\)-module algebra.

With the actions \(\gamma_{\pm}\) of \(\extq(\upm)\) on \(\extq(\up)\) as above, we obtain a \(\uql\)-equivariant map
\begin{equation}
  \label{eq:gamma-q-def}
  \gamma : \extq(\um) \otimes \extq(\up) \to \End_\bbC(\extq(\up)), \quad y \otimes x \mapsto \gamma_-(y) \gamma_+(x).
\end{equation}
This map is a quantization of \eqref{eq:factorization-of-clifford-algebra}.
That it is an isomorphism of \(\uql\)-modules follows from the fact that the two factors are Frobenius algebras:

\begin{thm}
  \label{thm:spinors-are-irreducible}
  The map \(\gamma\) is an isomorphism.
\end{thm}

\begin{proof}
  We show that \(\gamma\) is injective, and hence for dimension reasons is an isomorphism.
  Indeed, assume that
  \[
  c = \sum_{I \subseteq [N]} c_I \otimes x_I
  \]
  lies in the kernel of \(\gamma\) for some elements \(c_{I} \in \extq(\um)\).
  Here \(\{ x_I \}\) is the basis of \(\extq(\up)\) introduced in \cref{dfn:basis-of-quantum-exterior-algebra}.
   By \cref{prop:frobenius-for-quantum-exterior-algebras}(b), there is a unique ``dual basis'' \(\{ z_J \}_{J \subseteq [N]}\) for \(\extq(\up)\), with \(\deg z_J = N - \abs{J}\), which satisfies
   \begin{equation}
     \label{eq:dual-basis-properties}
     x_I \wedge z_J =
     \begin{cases}
       0 & \text{ if } \abs{I} > \abs{J}\\
       \delta_{IJ} x_{[N]} & \text{ if } \abs{I} = \abs{J}
     \end{cases},
   \end{equation}
   while if \(\abs{I} < \abs{J}\) then \(x_I \wedge z_J\) lies in \(\extq^{< N}(\up)\).
   Applying \(\gamma(c) = \sum_{I} \gamma_-(c_I) \gamma_+(x_I)\) to \(z_J\) and using \eqref{eq:dual-basis-properties}, we obtain
   \begin{equation}
     \label{eq:proof-of-gamma-isomorphism}
     0 = \sum_{I \subseteq [N]} \gamma_-(c_I) x_I \wedge z_J = \gamma_-(c_J) x_{[N]} + \sum_{\abs{I} < \abs{J}} \gamma_-(c_I) x_I \wedge z_J.
   \end{equation}

   Now we claim that if \(\gamma_-(y)x_{[N]} = 0\) for some \(y \in \extq(\um)\), then \(y = 0\).
   Indeed, if \(\gamma_-(y)x_{[N]} = 0\) then for any \(w \in \extq(\um)\) we have
   \[
   0 = \langle w, \gamma_-(y) x_{[N]} \rangle \eqdef \langle w \wedge y, x_{[N]} \rangle.
   \]
   However, as \(\extq(\um)\) is a graded Frobenius algebra, and in light of \cref{prop:frobenius-for-quantum-exterior-algebras}(b), if \(y \neq 0\) then there is an element \(w \in \extq(\um)\) such that \(w \wedge y = y_{[N]}\), so we get \(\langle y_{[N]}, x_{[N]} \rangle = 0\).
   This contradicts \cref{prop:dual-of-exterior-algebra}; hence we must have \(y = 0\).
   
   Applying this claim together with induction on \(\abs{J}\) to \eqref{eq:proof-of-gamma-isomorphism}, we conclude that \(c_J = 0\) for all \(J\), and hence \(c = 0\).
   Thus \(\gamma\) is injective.
\end{proof}

\begin{dfn}
  \label{dfn:clifford-alg-def}
  We define the \emph{quantum Clifford algebra} to be 
\(\clifq \eqdef \End_\bbC(\extq(\up))\) together with the factorization \(\gamma : \extq(\um) \otimes \extq(\up) \to \clifq\) from 
\cref{thm:spinors-are-irreducible}. 
\end{dfn}

\begin{rem}\label{rem:on-the-cross-relations}
Several authors have considered quantum Clifford algebras previously, e.g.\ \cite{BauEtal96,BrzPapRem93,Fio98,Fio99,Han00,Hec03}.
In these approaches the Clifford algebras were defined explicitly by generators and relations, in contrast to our \cref{dfn:clifford-alg-def}.
Classically, the Clifford algebra of \(\up \oplus \um\) is defined abstractly as a quotient of the tensor algebra.
The isomorphism \eqref{eq:factorization-of-clifford-algebra} is then easily established just using the cross-relations between the creation and annihilation operators,
\[ 
\gamma_+(x) \gamma_-(y) + 
\gamma_-(y) \gamma_+(x) = \langle
x,y \rangle,\quad y \in \um, x \in \up.  
\]
While the existence of such a presentation of \(\clifq\) 
can be deduced from the factorization \(\gamma\),
the cross-relations are in general more complicated and 
involve terms of higher order.
\end{rem}

\subsection{The Dolbeault operator \(\eth\)}
\label{sec:koszul-complex-of-symq-uplus}

We will now embed the Koszul differential for \(\symq(\up)\) into \(\uqg \otimes \clifq\).
Recall that this is the canonical element in \(\up \otimes \um\), viewed as an element in \(\symq(\up)^{\op} \otimes \extq(\um)\):

\begin{dfn}
  \label{dfn:dolbeault-operator}
  We define
  \[
  \eth \eqdef \sum_{i = 1}^N  x_i \otimes y_i \in \symq(\up)^{\op} \otimes \extq(\um),
  \]
  where \(\{ x_i \}\) and \(\{ y_i \}\) are the bases for \(\up\) and \(\um\) defined in \cref{sec:quantum-exterior-algebras-are-frobenius}.
\end{dfn}

As \(\{ x_i \}\) and \(\{ y_i \}\) are dual bases, the element \(\eth\) is independent of their choice.
The following result is well known:
\begin{prop}
  \label{prop:properties-of-delbar}
  With respect to the ordinary tensor product algebra structure on \(\symq(\up)^{\op} \otimes \extq(\um)\), we have \(\eth^2 = 0\).
\end{prop}

\begin{proof}
  See, for example, Chapter~2, Section~3 of \cite{PolPos05}.
  Recall from \cref{rem:on-def-of-koszul-dual} that the quadratic dual algebra in our conventions is the opposite of the one defined there.
  That is, when we consider the Koszul complex providing a resolution of \(\bbC\) as a left \(\symq(\up)\)-module, then \(\eth\) acts as
  \[
  \eth (a \otimes f) = \sum_{i=1}^N a x_i \otimes y_i f, \quad a \otimes f \in \symq(\up) \otimes \extq(\um)^*. \qedhere
  \]
\end{proof}

Next we embed \(\eth\) into \(\uqg \otimes \clifq\).
Recall that
the quantum root vectors \(E_{\xi_i}\) are the 
generators of the twisted quantum Schubert cell 
\(U'(\wl)\) as in \cref{sec:twisted-quantum-schubert-cells}.
\begin{lem}\label{lem:embed}
  The assignment 
  \[  
  x_i \otimes y \mapsto 
  S^{-1}(E_{\xi_i}) \otimes \gamma_-(y),\quad
  y \in \extq(\um)
  \]
  extends to an algebra embedding \(\iota : \symq(\up)^{\op} \otimes \extq(\um) \rightarrow \uqg \otimes \clifq\). 
\end{lem}

\begin{proof}
  By construction, the linear map \(x_i \mapsto E_{\xi_i}\) extends to an algebra isomorphism from \(\symq(\up)\) to \(U'(\wl) \subseteq \uqg\); see \cref{conv:generators-of-quantum-symmetric-algebra}.
  The inverse of the antipode is an anti-automorphism of \(\uqg\).
  Finally, \(\gamma_-\) is an algebra homomorphism by definition.
\end{proof}

By slight abuse of notation, we denote \(\iota(\eth)\) 
also by \(\eth\).

\subsection{\(*\)-structures and Dirac operators}
\label{sec:star-struct-and-dirac-operator}

Recall from \cref{sec:enveloping-alg} that \(\uqg\) has a 
\(\ast\)-structure called the compact real form. Choosing any
\(*\)-structure on \(\clifq\) leads to:
\begin{dfn}
  \label{dfn:dolbeault-dirac-operator}
  We define the \emph{Dolbeault--Dirac operator} 
  \[
  D \eqdef \eth + \eth^\ast \in \uqg \otimes \clifq.
  \]
\end{dfn}

The relation to the classical Dolbeault--Dirac operator from complex geometry 
will be explained in \cref{sec:motivation-and-outlook} below. 
With this definition, \cref{thm:main-theorem}
follows immediately from
\cref{prop:properties-of-delbar} and
\cref{lem:embed}. 

\begin{rem}
  \label{rem:on-the-star-structure}
  The element \(\eth^\ast\) and hence \(D\) depend on the choice of the 
  \(\ast\)-structure on \(\clifq\).
  \cref{thm:main-theorem} holds regardless of this
  choice, but the choice matters in the potential
  applications that we outline in
  \cref{sec:motivation}.  
  
  For these, the \(\ast\)-structure must arise from a
  \(\uql\)-invariant Hermitian inner product on
  \(\extq(\up)\).    
  On \(\up\) this is unique up
  to a positive scalar factor, as \(\up\) is irreducible.
  This extends canonically  to each tensor power
  \(\up^{\otimes k}\) by \[ \langle x_1 \otimes \dots
  \otimes x_k, z_1 \otimes \dots \otimes z_k \rangle =
  \langle x_1, z_1 \rangle \cdots \langle x_k, z_k
  \rangle, \] and then we can restrict this inner product
  to the submodule \(\extq^k \up\).  According to
  Proposition~3.2 of \cite{ChiTuc12}, the quotient map
  \(\up^{\otimes k} \to \extq^k(\up)\) restricts to an
  isomorphism \(\extq^k \up
  \overset{\sim}{\longrightarrow} \extq^k(\up)\) of
  \(\uql\)-modules; this gives us an invariant inner
  product on each \(\extq^k(\up)\) and hence on 
  \(\extq(\up)\). 

  Among the \(\uql\)-equivariant \(\ast\)-structures on
  \(\clifq\), the ones constructed in this manner are
  distinguished by the fact that 
  \(\gamma_+(x_i)^\ast = \gamma_-(y_i)\) as long as the
  \(x_i\) are orthonormal with respect to the chosen
  inner product on \(\up\).
\end{rem}

\section{An example: $\Gr(2,4)$}
\label{sec:grassmannian}

%

In this section we present the details of the preceding constructions
in the case when the Dynkin diagram with crossed node is 
$\, \DynkinDiagram{\bullet \Edge{r} & \times \Edge{r} & \bullet}$.
This means that $\fg = \fs\fl_4$ and $\fp$ consists of all matrices in $\fg$ of the form
$$
\begin{pmatrix}
*&*&*&*\\
*&*&*&*\\
0&0&*&*\\
0&0&*&*
\end{pmatrix}.
$$
The associated irreducible flag manifold is the
Grassmann manifold $\Gr(2,4)$ of all 2-dimensional
subspaces of $ \mathbb{C}^4$. This example is interesting 
for its role in Yang--Mills theory (see
e.g.~\cite{Ati79,BasEas89} and the references therein),
and since the quantum Dolbeault--Dirac operator (see
also the final section below) has not been constructed
explicitly in the literature yet. 

\subsection{Lie algebras and root data}
\label{sec:liealgs}

Let $\alpha_1,\alpha_2,\alpha_3$ be the simple roots of 
$\fg=\mathfrak{sl}_4$ (using the conventions
of \cite[\S 11.4]{Hum90}), and take $s =2$, so that 
$\cS = \{ \alpha_1, \alpha_3  \}$.
The Levi factor $\fl$ of $\fp$ is given by 
$$\fl \cong \fs(\fg\fl_2 \times \fg\fl_2) = \{ (X,Y) \in \fg\fl_2  \times \fg\fl_2 \mid  \tr X + \tr Y = 0 \},$$
so $\fk \cong \fs\fl_2 \times \fs\fl_2$,
and the nilradical $\up$ of $\fp$ is an abelian Lie algebra, embedded
as the upper-right $2 \times 2$ block of $\fp$, so
the decomposition $\fs\fl_4 \cong \um \oplus \fl \oplus \up$, in $2 \times 2$ block form,  is given by
$$
\fs\fl_4 = 
\begin{pmatrix}
  \fl & \up \\
  \um & \fl
\end{pmatrix},
$$
and the action of $\fl$ on $\up$ is given by
$$
(X,Y)\cdot Z = XZ - ZY,
$$
where $X,Y,Z$ are all $2 \times 2$ matrices, viewed as embedded in $\fs\fl_4$ as the upper left, lower right, and upper right blocks, respectively.

The Weyl group of $\fs\fl_4$ is $W \cong \mathfrak{S}_4$, the permutation
group on four letters, and we choose the reduced decomposition 
$$
w_0 = w_{0,\fl} w_\fl = (s_1 s_3) (s_2 s_3 s_1 s_2)
$$
of the longest word.
In the notation of \cref{sec:quantum-schubert-cell}, we have
$W_\fl = \langle s_1, s_3 \rangle \cong \bbZ_2 \times \bbZ_2$,
$w_{0,\fl} = s_1 s_3$, and the parabolic element $w_\fl$ is given by
$$
w_\fl = w_{0,\fl} w_0 = s_2 s_3 s_1 s_2.
$$

\subsection{The quantized enveloping algebra}
\label{sec:qalg}

The sets $\{ E_1,F_1,K_1 \}$ and $\{ E_3,F_3,K_3 \}$ generate commuting copies of $\uqsltwo$ inside $\uqslfour$, and we denote 
\begin{equation}
\label{uqkdecomp}
\uqk = \langle E_1,F_1,K_1,E_3,F_3,K_3 \rangle \simeq \uqsltwo \otimes \uqsltwo \subset \uqslfour.
\end{equation}

\subsection{The quantized representation $\up$}
\label{sec:qrep}

Let $V = \spn \{ v_1, v_2 \}$ be the two-dimensional irreducible representation 
of $\uqsltwo$, where $v_1$ is a highest weight vector.
The generators $E,F,K$ of $\uqsltwo$ act in $V$ via the matrices
\[
e =
\begin{pmatrix}
  0 & 1 \\ 0 & 0
\end{pmatrix},
\quad
f = 
\begin{pmatrix}
 0 & 0 \\ 1 & 0
\end{pmatrix}, 
\quad
k = 
\begin{pmatrix}
  q & 0 \\ 0 & q^{-1}
\end{pmatrix}.
\]

The representation $\up$ of $\uqk$ is given by 
\[
\up \cong V \otimes V
\]
with respect to the decomposition \eqref{uqkdecomp}.
This means that the first copy of $\uqsltwo$ acts in $\up$ via
\[
E_1 \mapsto e \otimes \id_V, \quad F_1 \mapsto f \otimes \id_V, \quad K_1 \mapsto k \otimes \id_V,
\]
while the second copy acts via
\[
E_3 \mapsto \id_V \otimes e, \quad F_3 \mapsto   \id_V \otimes f, \quad K_3 \mapsto  \id_V  \otimes k.
\]

\subsection{The coboundary structure}
\label{sec:coboundaryop}

The commutor $\sigma_{VV} : V \otimes V \to V \otimes V$ for the two-dimensional representation of $\uqsltwo$ is given by
\[
\sigma_{VV} =
\begin{pmatrix}
 1 & 0 & 0 & 0 \\
 0 & \frac{q^2-1}{1+q^2} & \frac{2 q}{1+q^2} & 0 \\
 0 & \frac{2 q}{1+q^2} & \frac{1-q^2}{1+q^2} & 0 \\
 0 & 0 & 0 & 1 \\
\end{pmatrix}
\]
with respect to the (lexicographically ordered) tensor product basis
\[
x_1 \eqdef v_1 \otimes v_1, \quad x_2 \eqdef v_1 \otimes v_2, \quad x_3 \eqdef v_2 \otimes v_1, \quad x_4 \eqdef v_2 \otimes v_2
\]
for $V \otimes V$.
The commutor for $\up = V \otimes V$ is then given by
\[
 \sigma_{\up\up} = \tau_{23} \circ (\sigma_{VV} \otimes \sigma_{VV}) \circ \tau_{23} : \up \otimes \up \to \up \otimes \up,
\]
where $\tau_{23}$ is the tensor flip in the second and third components when we make the identification $\up \otimes \up \cong V^{\otimes 4}$.

\subsection{The quantum symmetric and exterior algebras}
\label{sec:qsymextalgs}

Diagonalizing \(\sigma_{\up\up}\), we find that the space 
of  quantum antisymmetric 2-tensors
\(\extq^2 \up = \ker (\sigma_{\up\up} + \id)\), i.e., the space of relations for
 the quantum symmetric algebra,  is the span of
\[
\begin{gathered}
x_1\otimes x_2 - qx_2\otimes x_1, \quad
x_1\otimes  x_3 - qx_3\otimes x_1 ,\\
x_2\otimes  x_4 - qx_4\otimes x_2 ,\quad
x_3\otimes x_4 - qx_4\otimes x_3  ,\\
x_2\otimes x_3 - x_3\otimes x_2, \quad
x_1\otimes x_4 - x_4\otimes x_1 -(q-q^{-1}) x_2\otimes x_3.
\end{gathered}
\]
The corresponding relations are
\[
\begin{gathered}
x_1  x_2 = qx_2  x_1, \quad
x_1   x_3 = qx_3  x_1 ,\\
x_2   x_4 = qx_4  x_2 ,\quad
x_3  x_4 = qx_4  x_3  ,\\
x_2  x_3 = x_3  x_2, \quad
x_1  x_4 - x_4  x_1 =(q-q^{-1}) x_2  x_3.
\end{gathered}
\]
Hence, the quantum symmetric algebra \(\symq(\up)\) is the familiar 
algebra of \(2 \times 2\) quantum matrices.

The quantum exterior algebra \(\extq(\um)\) is the quadratic dual of \(\symq(\up)\)
(keeping in mind \cref{rem:on-def-of-koszul-dual}). Denoting the dual basis
to \(\{ x_1, x_2, x_3, x_4 \}\) by \(\{ y_1,y_2,y_3,y_4 \}\) and 
the multiplication in \(\extq(\um)\) by \(\wedge\), we find that
 the relations are
\[
\begin{gathered}
  y_1 \wedge y_1 =   y_2 \wedge y_2 =   y_3 \wedge y_3 =   y_4 \wedge y_4 = 0,\\
y_2 \wedge y_1 = -q^{-1} y_1 \wedge y_2, \quad
y_3 \wedge y_1 = -q^{-1} y_1 \wedge y_3,\\
y_4 \wedge y_2 = -q^{-1} y_2 \wedge y_4, \quad
y_4 \wedge y_3 = -q^{-1} y_3 \wedge y_4,\\
y_4 \wedge y_1 = - y_1 \wedge y_4, \quad
y_2 \wedge y_3 + y_3 \wedge y_2 = (q-q^{-1}) y_1 \wedge y_4.
\end{gathered}
\]

\subsection{The twisted quantum Schubert cell}
The radical roots are 
$$
	\xi_1=\alpha_1+\alpha_2+\alpha_3,\quad
	\xi_2=\alpha_1+\alpha_2,\quad
	\xi_3=\alpha_2+\alpha_3,\quad
	\xi_4=\alpha_2,
$$ 
and the corresponding quantum root vectors, defined
with respect to the expression for $w_0$ given above, 
that generate the 
twisted quantum Schubert cell are
\[
\begin{gathered}
	E_{\xi_1}=
	E_3(E_1E_2-q^{-1}E_2E_1)-
   q^{-1}(E_1E_2-q^{-1}E_2E_1)E_3,\quad
	E_{\xi_2}=
	E_1E_2-q^{-1}E_2E_1,\\
	E_{\xi_3}=
	E_3E_2
   -q^{-1} E_2E_3,\quad
	E_{\xi_4}=E_2.
\end{gathered}
\] 
\cref{prop:zwicknagl-main-theorem-5-6} can
now be verified directly 
 for this example
using the quantum Serre relations; that is, 
the elements $E_{\xi_1},E_{\xi_2},E_{\xi_3},E_{\xi_4}
\in U_q(\fg)$ indeed satisfy the
same relations as $x_1,x_2,x_3,x_4$ and transform in
the same way as them under
the adjoint action of $U_q(\fl)$.

Although the span of the \(E_{\xi_i}\) is invariant under the adjoint
action of \(\uql\), it is not in the locally finite part of \(\uqg\), 
i.e., its orbit under the  adjoint action of all of \(\uqg\) is
infinite-dimensional (see \S7.1.3 of \cite{Jos95} for more information).
If one multiplies the \(E_{\xi_i}\) on the right by \(K_{-2n\omega_2}\), 
the span of the resulting elements is in the locally finite part of
\(\uqg\). These elements coincide with the \(X^i\) defined 
in \S4 of \cite{Kra04}.

\section{A motivation}
\label{sec:motivation}

This final section aims at explaining our motivation for this paper, which was to further study the spectral triples on the quantized compact Hermitian symmetric spaces from \cite{Kra04}.
Before we discuss this application of the results, we explain the classical picture, that is, how the counterparts of our algebraic structures relate for \(q=1\) to the geometry of the compact Hermitian symmetric spaces.

\subsection{The classical geometric picture}
\label{sec:motivation-and-outlook}

First of all, view a pair of a complex semisimple Lie algebra \(\fg\) and a para\-bolic Lie subalgebra \(\fp\) as an infinitesimal description of the complex manifold \(G/P\), where \(G\) is the (connected, simply connected) Lie group corresponding to \(\fg\) and \(P\) is the parabolic subgroup having Lie
algebra \(\fp\).
These spaces are referred to as the \emph{generalized flag manifolds}, and (with respect to a Hermitian metric induced by the Killing form of \(\fg\)) they exhaust the compact homogeneous K\"ahler manifolds \cite{Wan54} as well as the coadjoint orbits of the compact semisimple Lie groups.
This leads to a wealth of applications in geometry, physics, and representation theory; see e.g.~\cite{BasEas89,ChrGin10}, \cite[Chapter~8]{Bes08}.   
The case in which \(\fp\) is of cominuscule type as in \cref{prop:irred-parabolic-conditions} corresponds to \(G/P\) being a \emph{symmetric space}; see e.g.~\cite{Kos61}*{Proposition~8.2}. 
This symmetric space is irreducible precisely when \(\fg\) is simple, so the pairs \((\fg,\fp)\) that we consider throughout the paper correspond to the irreducible compact Hermitian symmetric spaces.
A general compact Hermitian symmetric space is just a product of irreducible ones. 

As a real manifold, \(G/P\) is diffeomorphic to \(G_0/L_0\), where \(G_0\) is the compact real form of \(G\) and \(L_0 = L \cap G_0\) is the compact real form of \(L\) \cite{BasEas89}*{\S 6.4}. 
If \(Q\) is the parabolic subgroup of \(G\) with Lie algebra \(\fq = \fl \oplus \um\), then \(G/Q\) is also diffeomorphic to \(G_0/L_0\) and hence to \(G/P\).
However, the two induced complex structures on \(G_0/L_0\) are inverse to each other.

Our next aim is to describe the Dolbeault complex \((\Omega^{(0,\bullet)}, \delbar)\) of the complex manifold \(G/Q\). 
To this end, identify \(\fg/\fq\) with the holomorphic tangent space of \(G/Q\) at the identity coset.
This identifies the adjoint representation of \(\fq\) on \(\fg/\fq\) with the isotropy representation.
As representations of \(L_0\) we have \(\fg/\fq \cong \up\).
Hence the smooth sections of the holomorphic tangent bundle \(T^{(1,0)}\) can be identified with the \(L_0\)-equivariant smooth functions
\begin{equation}
  \psi \colon G_0 \rightarrow \up,\quad
  \psi(gh)=h^{-1} \psi(g) \quad \text{for all } g \in G_0, h \in L_0,\label{eq:sections-of-t10}
\end{equation}
as \(T^{(1,0)}\) is associated to the \(L_0\)-principal fiber bundle \(G_0 \to G_0 / L_0\) by the isotropy representation.
We fix an \(L_0\)-invariant Hermitian inner product \(\langle \cdot,\cdot \rangle\) on \(\up\), which induces an isomorphism of complex vector bundles \(T^{(1,0)} \cong \Omega^{(0,1)}\).
Hence from now on we view functions \(\psi\) as in \eqref{eq:sections-of-t10} as smooth \((0,1)\)-forms on \(G/Q\).
Similarly, \((0,n)\)-forms are identified with \(L_0\)-equivariant smooth functions from \(G_0\) to \(\ext^n(\up)\).
Finally, the \emph{Dolbeault operator}
\begin{equation}
  \label{eq:delbar-definition}
  \delbar : \Omega^{(0,n)} \rightarrow 
  \Omega^{(0,n+1)}
\end{equation}
acting on these functions is obtained by taking Cartan's differential \(\mathrm{d}\) of a \((0,n)\)-form and projecting onto \(\Omega^{(0,n+1)}\).

In order to construct the Hilbert space of square-integrable sections of the bundle \(\Omega^{(0,\bullet)}\), one embeds \(\ext^n(\up)\) into \(\up^{\otimes n}\) and extends the inner product on \(\up\) to one on \(\ext^n(\up)\), for each \(n\), as in \cref{rem:on-the-star-structure}.
Then one obtains an inner product on smooth sections of \(\Omega^{(0,\bullet)}\), defined for \(L_0\)-equivariant functions \(\phi,\psi : G_0 \to \ext(\up)\) by
\begin{equation}
  \label{eq:inner-product-on-forms}
  (\phi,\psi) \eqdef \int_{G_0} \langle \phi(g), \psi(g) \rangle dg,
\end{equation}
where we integrate with respect to the Haar measure of \(G_0\).
We denote by \(\cH\) the Hilbert space completion of \(\Omega^{(0,\bullet)}\) with respect to this inner product.

The universal enveloping algebra \(U(\fg)\) acts on the algebra \(C^\infty(G_0)\) of smooth complex-valued functions on \(G_0\) by extending the action of \(U(\fg_0)\) by differential operators.
If \(\cliff\) is the Clifford algebra of \(\up \oplus \um\) with respect to the canonical symmetric bilinear form, then \(\cliff\) acts naturally on \(\ext(\up)\) via the representation constructed in \cref{sec:classical-clifford-alg}.
Hence
\[
U(\fg) \otimes \cliff  
\]
acts on the algebra \(C^\infty(G_0) \otimes \ext(\up)\) of all smooth functions from \(G_0\) to \(\ext(\up)\).
Under this action, the classical analogue of our \(\eth\) is easily seen to leave the subalgebra of \(L_0\)-equivariant functions invariant, and we have
\begin{equation}
  \label{eq:delbar-eth-adjoint}
  (\delbar \phi, \psi) = (\phi, \eth \psi)
\end{equation}
for smooth sections \(\phi,\psi\) of \(\Omega^{(0,\bullet)}\).

The choice of compact real form of \(G\) induces a \(\ast\)-structure on \(U(\fg)\), and the Hermitian inner product on \(\ext(\up)\) induces a \(\ast\)-structure on \(\cliff\).
By \eqref{eq:delbar-eth-adjoint} the element
\[
D \eqdef \eth + \eth^\ast \in U(\fg) \otimes \clif
\]
acts, up to a scalar, as the Dolbeault--Dirac operator on \(G/Q\) formed with respect to the canonical spin\(^c\)-structure; see \cite{Fri00}*{\S 3.4} 
or \cite{BerGetVer04}*{\S 3.6}.
Analogously, \(D\) implements the Dolbeault--Dirac operator of \(G/P\) formed with respect to the anti-canonical spin\(^c\)-structure.

The point of the algebraic description of the 
Dolbeault--Dirac operator is that it leads
to a computation of its square and of its spectrum,
based on the celebrated Parthasarathy formula that
expresses \(D^2\) as a linear combination of Casimir
elements in \(U(\fg),U(\fl)\) and
constants; see \cite{Par72,Agr03,Kos99} and also \cite[Lemma~12.12]{Kna01}
for this algebraic approach to Dirac operators and the
Parthasarathy formula, \cite{CahFraGut89,CahGut88,Rie09,Sem93}
for the construction of spinors on symmetric spaces and
the application of Parthasarathy's formula in explicit
computations of spectra.
As we will explain next, 
the present article is meant as a step toward
a quantum analogue of these results.

\subsection{Spectral triples on quantized \(G/P\)}
\label{sec:quantum-motivation}

Recall that the matrix coefficients 
of the Type 1 representations of \(\uqg\) 
generate a Hopf \(\ast\)-algebra \(\bbC_q[G]\) that deforms the complex
coordinate ring of the real affine algebraic group \(G_0\).
The universal \(C^*\)-completion of \(\bbC_q[G]\)
is the fundamental example of 
a compact quantum group in the sense of Woronowicz \cite{Wor87}. 
Together with \(G_0\), one can quantize 
\(L_0\) in the form of a quotient Hopf \(\ast\)-algebra 
\(\bbC_q[L]\), and also \(G_0/L_0\) in the form of a right coideal subalgebra 
\(A\) of \(\bbC_q[G]\) \cite{Dij96,NouSug94,MulSch99}.
Associated vector bundles such as 
\(\Omega^{(0,\bullet)}\) can be quantized in the form of finitely generated
projective \(A\)-modules, which admit  
Hilbert space completions \(\cH\) using 
the Haar measure of the 
\(C^\ast\)-completion of \(\bbC_q[G]\). 
See e.g.~\cite{GovZha99,Kra04} and the references therein 
for these topics. 
The paradigmatic example of such a quantized symmetric space is the standard Podle\'s quantum sphere \cite{Pod87}. 

These structures all arise naturally from quantum group theory.
An obvious question to ask is whether there is also a Dirac-type operator \(D\) on a quantized spinor module that produces a spectral triple \((A,\cH,D)\) in the sense of Connes.
This would provide a quantization of the metric structure of \(G/P\).

D\k{a}browski and Sitarz constructed such a spectral triple over the Podle\'s sphere by deforming the Dirac operator with respect to the standard spin structure and Levi-Civita connection \cite{DabSit03}.
In \cite{Kra04}, an abstract argument was given that a quantization of the
Dolbeault--Dirac operator on all symmetric \(G/P\) exists.
It was shown that the commutators \([D,a]\) between algebra elements \(a \in A\) and the Dirac operator are given by bounded operators, which is the first axiom of a spectral triple.
However, the implicit nature of the construction meant that it was not possible to compute the spectrum of \(D\), nor even to prove that \(D\) had compact resolvent.
The latter is the second axiom for a spectral triple, and is a key condition for it to define a K-homology class for the \(C^*\)-algebra completion of 
the quantization \(A\) of \(G_0/L_0\).
Up to now, the only cases in which this has been carried out are the projective spaces; see \cite{DanDabLan08}.
The approach is by direct computation, and relies on the Hecke condition for the relevant braidings, so it seems difficult to generalize these methods to arbitrary \(G/P\).

The main motivation for us is that our new approach to the construction of \(\clifq\), and hence a quantization of the Dolbeault--Dirac operator in terms of the algebras of Berenstein and Zwicknagl, might lead to new techniques for its study and ultimately to a quantum version of the Parthasarathy formula.

Another natural problem that arises is whether one can construct
nonstandard quantized enveloping algebras and symmetric spaces also
for other deformations of polynomial rings, such as the Jordan plane or the Sklyanin algebras.


\end{document}